\documentclass[a4paper, bibliography=totoc]{scrartcl}
\usepackage{amsmath}
\usepackage{amssymb}
\usepackage{enumerate}
\usepackage{amsthm}
\usepackage{todonotes}
\usepackage{mathtools} 
\usepackage[UKenglish]{babel}
\usepackage[T1]{fontenc}
\usepackage[utf8]{inputenc}
\usepackage{lmodern} 
\usepackage{tikz}
\usepackage{thmtools} 
\usepackage{thm-restate} 
\usepackage{stackrel} 
\usepackage[mathlines]{lineno}
\usepackage[normalem]{ulem}
\usepackage{graphicx}

\graphicspath{{./figs}}

\usepackage{scalerel}
\usepackage{tikz}
\usetikzlibrary{svg.path}
\definecolor{orcidlogocol}{HTML}{A6CE39}
\tikzset{
	orcidlogo/.pic={
		\fill[orcidlogocol] svg{M256,128c0,70.7-57.3,128-128,128C57.3,256,0,198.7,0,128C0,57.3,57.3,0,128,0C198.7,0,256,57.3,256,128z};
		\fill[white] svg{M86.3,186.2H70.9V79.1h15.4v48.4V186.2z}
		svg{M108.9,79.1h41.6c39.6,0,57,28.3,57,53.6c0,27.5-21.5,53.6-56.8,53.6h-41.8V79.1z M124.3,172.4h24.5c34.9,0,42.9-26.5,42.9-39.7c0-21.5-13.7-39.7-43.7-39.7h-23.7V172.4z}
		svg{M88.7,56.8c0,5.5-4.5,10.1-10.1,10.1c-5.6,0-10.1-4.6-10.1-10.1c0-5.6,4.5-10.1,10.1-10.1C84.2,46.7,88.7,51.3,88.7,56.8z};
	}
}

\newcommand\orcidicon[1]{\href{https://orcid.org/#1}{\mbox{\scalerel*{
				\begin{tikzpicture}[yscale=-1,transform shape]
					\pic{orcidlogo};
				\end{tikzpicture}
			}{Q}}}}
		
\usepackage[unicode]{hyperref} 
\definecolor{blue3}{rgb}{.1,.0,.4}
\hypersetup{colorlinks=true, linkcolor=red, urlcolor=blue3, citecolor=blue3, pdfpagemode=UseNone, pdfstartview=FitH, bookmarksopen=true, hypertexnames=false} 
\usepackage[open]{bookmark}

\declaretheorem[name=Theorem,numberwithin=section]{thm} 
\newtheorem*{thm*}{Theorem}

\newtheorem*{define*}{Definition}
\newtheorem{define}[thm]{Definition}

\newtheorem*{lemma*}{Lemma}
\newtheorem{lemma}[define]{Lemma}

\newtheorem*{algorithm*}{Algorithm}

\newtheorem*{notation*}{Notation}

\newtheorem*{construction*}{Construction}

\newtheorem*{prop*}{Proposition}
\newtheorem{prop}[define]{Proposition}

\newtheorem*{obs*}{Observation}

\newtheorem*{fact*}{Fact}

\newtheorem*{remark*}{Remark}

\newtheorem*{quest*}{Question}

\newtheorem*{cor*}{Corollary}
\newtheorem{cor}[define]{Corollary}

\newtheorem*{conjecture*}{Conjecture}

\newtheorem*{question*}{Question}

\newtheorem*{example*}{Example}

\newcounter{claimcounter}[define]
\numberwithin{claimcounter}{define}
\newtheorem*{claim*}{Claim}
\newtheorem{claim}[claimcounter]{Claim}

\newcommand{\R}{\mathbb{R}}

\newcommand{\Z}{\mathbb{Z}}
\newcommand{\N}{\mathbb{N}}

\newcommand{\floor}[1]{\lfloor#1\rfloor}
\newcommand{\ceil}[1]{\lceil#1\rceil}

\newcommand{\rest}[2]{#1|_{#2}} 
\newcommand{\lip}{\operatorname{Lip}}
\newcommand{\bilip}{\operatorname{bilip}}
\DeclareMathOperator{\diam}{diam}

\DeclareMathOperator{\proj}{proj}

\DeclareMathOperator{\dist}{dist}
\newcommand{\abs}[1]{\left|#1\right|}
\newcommand{\lnorm}[2]{\left\|#2\right\|_#1}
\newcommand{\infnorm}[1]{\left\|#1\right\|_\infty}
\newcommand{\enorm}[1]{\left\|#1\right\|}
\newcommand{\norm}[1]{\left\|#1\right\|}
\newcommand{\opnorm}[1]{\lnorm{\text{op}}{#1}}

\newcommand{\mc}[1]{\mathcal{#1}}
\newcommand{\mf}[1]{\mathfrak{#1}}
\newcommand{\mb}[1]{\mathbf{#1}}

\newcommand{\set}[1]{\left\{#1\right\}}
\newcommand{\br}[1]{\left(#1\right)}
\newcommand{\sqbr}[1]{\left[#1\right]}

\DeclareMathOperator{\dom}{dom}
\DeclareMathOperator{\img}{image}

\DeclareMathOperator{\inter}{Int}
\newcommand{\wt}[1]{\widetilde{#1}}

\newcommand{\cl}[1]{\overline{#1}}

\makeatletter
\newcommand{\mylabel}[2]{#2\def\@currentlabel{#2}\label{#1}}

\makeatother 

\title{Planar Bilipschitz Extension from Separated~Nets}
\author{Michael Dymond\thanks{The present work developed from a research visit of M.D. to V.K. at IST Austria, funded by a London Mathematical Society Research in Pairs grant.} \\ \small{School of Mathematics}
	\\ \small{University of Birmingham}
	\\ \small{Birmingham B15 2TT, UK}
	\\ \small{\orcidicon{0000-0002-1900-3549}\ \href{https://orcid.org/0000-0002-1900-3549}{0000-0002-1900-3549}
	}	\and	Vojtěch Kaluža\thanks{This work was done while V.K. was fully funded by the Austria Science Fund (FWF) [M 3100-N].}
	\\ \small{Institute of Science and Technology Austria}
	\\ \small{Am Campus 1}
	\\ \small{3400 Klosterneuburg, Austria}
	\\ \small{\orcidicon{0000-0002-2512-8698}\ \href{https://orcid.org/0000-0002-2512-8698}{0000-0002-2512-8698}
}}
\date{}

\begin{document}
	\maketitle
	\begin{abstract}
		We prove that every $L$-bilipschitz mapping $\Z^2\to\R^2$ can be extended to a $C(L)$-bilipschitz mapping $\R^2\to\R^2$ and provide a polynomial upper bound for $C(L)$. Moreover, we extend the result to every separated net in $\R^2$ instead of $\Z^2$, with the upper bound gaining a polynomial dependence on the separation and net constants associated to the given separated net. This answers an Oberwolfach question of Navas~\cite{adiceam2016open} from 2015 and is also a positive solution of the two-dimensional form of a decades old open (in all dimensions at least two) problem due to Alestalo, Trotsenko and Väisälä~\cite{bilip_ext_sep_nets}.
	\end{abstract}
	
	
	\section{Introduction.}
	
	A classical theorem of Kirszbraun~\cite{Kirszbraun} says that any $L$-Lipschitz mapping defined on $A\subseteq\R^m$ with values in $\R^n$ can be extended to an $L$-Lipschitz mapping $\R^m\to\R^n$. This is one of the fundamental results of Lipschitz analysis with a wide range of applications.
	
	While there is no obstruction for the extendability of Lipschitz mappings in the realm of euclidean spaces, the story is very different in the case of bilipschitz mappings. Every bilipschitz $f\colon A\subset\R^d\to\R^d$ is an embedding of $A$ into $\R^d$ and any bilipschitz mapping $\R^d\to\R^d$ is a homeomorphism. Thus, if $f$ does not preserve some topological invariant of the pair $(\R^d, A)$, then there is no hope to extend $f$ to a homeomorphism $\R^d\to\R^d$, let alone to a bilipschitz one. As an illustrative example, let $A$ be a unit circle in the plane together with its centre and consider $f$ that is equal to the identity on the circle, but maps its centre to a point outside the circle. Then there is no homeomorphic extension of $f$ to the whole plane, since every homeomorphism of the plane maps the inside of the circle onto the inside of the image of the circle.

	To add to the subtlety of the picture, even if there is a homeomorphic extension to $\R^{d}$ of a bilipschitz $f\colon A\to\R^d$, there still might not be a bilipschitz one; for an example, see \cite{Vaisala_bilip_and_QS}. 
	
	Another difference to the Lipschitz case stems from the fact that even if we \emph{know} that for $A\subseteq\R^d$ every bilipschitz $f\colon A\to\R^d$ can be extended to a bilipschitz mapping $\R^d\to\R^d$, it might not be possible for an extension to have the same bilipschitz constant as $f$; see, e.g., \cite[Ex.~5.2]{Kovalev_bilip_ext_from_line}. Controlling the bilipschitz constant of extensions (when they exist) therefore emerges as an additional facet to extension problems; see for example~\cite{DP_bilip_square,Kovalev_bilip_ext_from_line,Kovalev_optimal_bilip}.  
	
	For convenience in the continuing discussion, we fix some terminology: 
	\begin{center}
		\emph{The bilipschitz extension problem for a set $A\subseteq \R^{d}$ is the task of determining whether every bilipschitz mapping $f\colon A\to \R^{d}$ admits a bilipschitz extension $F\colon\R^{d}\to\R^{d}$. }
	\end{center}

	Naturally, investigation of the bilipschitz extension problem for $A\subseteq \R^{d}$ begins with the simplest forms of sets $A$. From a linear structure point of view these are lines and from a metric topology point of view, these are uniformly discrete, where the most basic example, beyond finite sets, is the integer lattice $A=\Z^{d}$. Whilst the bilipschitz extension problem for lines has a positive solution in $\R^{2}$~\cite{tukia1981extension} and a negative solution in $\R^3$ \cite[3.10]{Elems_of_Lip_topology}, \cite[14]{Tukia_planar_Schoenflies}, bilipschitz extension from the integer lattice $\Z^{d}$ has remained elusive. The bilipschitz extension problem for $\Z^{d}$, or, more generally, for separated nets\footnote{$A$ is a separated net in $\R^d$ if there are $r, R>0$ such that for every $x\in\R^d$ there is $a\in A$ such that $\enorm{x-a}\leq R$ and for every $a,b\in A, a\neq b$, $\enorm{a-b}\geq r$. These sets are also called \emph{Delone sets} in the literature.} of $\R^{d}$, is therefore fundamental for advancing understanding of bilipschitz extendability. 
	
	Accordingly, the bilipschitz extension problem for separated nets $A$ of $\R^{d}$ has received significant exposure and attention. Formally it dates back to 2003: Alestalo, Trotsenko and Väisälä~\cite[4.4(a)--(b)]{bilip_ext_sep_nets} pose the bilipschitz extension problem for $A=\Z^d$, and also, more generally, for the case where $A$ is a separated net in~$\R^d$. Both these questions have remained open in all dimensions at least two, despite the two dimensional form of the $A=\Z^{d}$ question being distinguished by Navas as one of the open problems~\cite[Problem~2.6.1]{adiceam2016open} considered by the Arbeitsgemeinschaft on Mathematical Quasicrystals, held in Oberwolfach in 2015. The problem has further appeared explicitly in \cite[Remark~7]{cortez2016some} and \cite[p.6205]{Smilansky_2022} in the context of rectifiability and bilipschitz equivalence of separated nets of $\R^{d}$. The latter topic has gained much investigation and research following, since the landmark papers of Burago and Kleiner~\cite{BK1} and McMullen~\cite{McM}, which show the existence of so-called non-rectifiable separated nets of $\R^{d}$. These are separated nets which do not permit any bilipschitz bijection to the integer lattice.
	
	In the present article, we provide a positive solution for the two dimensional bilipschitz extension problem for separated nets. In particular, this is a complete positive solution of Navas'~problem~\cite[Problem~2.6.1]{adiceam2016open} at Oberwolfach. It is also a positive solution of Alestalo, Trotsenko and Väisälä's~problems~\cite[4.4(a)--(b)]{bilip_ext_sep_nets} in dimension two. 
	
	Whenever the bilipschitz extension problem for a set $A\subseteq \R^{d}$ has a positive solution, this opens up a further question of the optimal bilipschitz constant with which an extension can be found. For applications, it is crucial to have a bound on the bilipschitz constant of the extension depending only on the bilipschitz constant of $f$, the set $A$ and the dimension $d$. Alestalo, Trotsenko and Väisälä~\cite[4.4(a)--(b)]{bilip_ext_sep_nets} specifically ask for this in their formulation of the bilipschitz extension problem for separated nets. 
	
	In addition to extending bilipschitz mappings of separated nets in $\R^{2}$, we will provide an explicit bound on the bilipschitz constant of our extension, determined by the bilip\-schitz constant of the initial mapping and the separation and net constants of the separated net.
	
	Most positive results on bilipschitz extension occur in dimension two. In two key settings, progress has started with a solution to a bilipschitz extension problem and then developed further to describe and improve on the control of the bilipschitz constant of the extension. Given a bilipschitz mapping to $\R^{2}$, defined on a boundary of a square or on a line in $\R^{2}$, there is a bilipschitz extension of the mapping to the whole square or to the whole of $\R^{2}$ respectively. Both of these results were first proven by Tukia~\cite{Tukia_planar_Schoenflies,tukia1981extension}, however, without providing any explicit dependence between the bilipschitz constant of the extension and that of the original mapping. An explicit bound in the case of the boundary of a square was first obtained by Daneri and Pratelli~\cite{DP_bilip_square}. Currently, the best dependence, due to Kovalev~\cite{Kovalev_optimal_bilip}, extends any $L$-bilipschitz mapping of the unit circle to a $2\cdot 10^{14}L$-bilipschitz mapping of the whole unit disk. An explicit bound has also been established in the case of the line: Kovalev~\cite{Kovalev_bilip_ext_from_line} proves that any $L$-bilipschitz mapping of a line in $\R^{2}$ to $\R^{2}$ can be extended to a $2000L$-bilipschitz mapping of $\R^{2}$.  
	
	The absence of positive results and significant additional difficulty of bilipschitz extension in dimensions higher than two is perhaps best demonstrated by the situation surrounding the Jordan--Schoenflies Theorem (see e.g. \cite[Thm.~3.1]{Thomassen-JordanSchoenflies}). The Jordan--Schoenflies Theorem states that any embedding of the circle in the plane extends to a homeomorphism of the plane. However, its analogue fails for embeddings of $S^{d-1}$ in $\R^{d}$ for all higher dimensions $d\geq 3$~\cite[Thm.~3.7]{Martin_wild_ball}, even if such an embedding, known as a wild sphere, is additionally assumed to be bilipschitz!~\cite[Thm.~7.2.1]{Gehring--Martin--Palka} We discuss the higher dimensional situation in greater detail in a companion article \cite{DK_ddim}. The main results of \cite{DK_ddim}, as well as the positive extension results discussed above, for the boundary of a square and for a line, play an important role in the present work. 
	
	There is a justifiable potential for the methodology and strategy introduced in the present solution of the two-dimensional bilipschitz extension problem for separated nets, to form a basis for the future research objective of achieving a solution in general higher dimension. This is substantiated in our companion article~\cite{DK_ddim}. In particular, in \cite[Theorem~1.5]{DK_ddim} we reduce the higher dimensional problem to two conjectures concerning bilipschitz extension to or from $(d-1)$-dimensional hyperplanes in $\R^{d}$. The proof of these two conjectures is completed in the present work in dimension $d=2$; significant and powerful new tools would need to be added in, in order to implement our strategies in higher dimensions.
	
	\paragraph{Acknowledgements.} 
	The authors wish to thank Professor Leonid Kovalev for a valuable observation on the first version of this work, which led to improved estimates and cleaner proofs in Section~\ref{ext_to_strip}.
	
	\subsection*{Main Results and Strategy.}
	The main result of the present article extends any $L$-bilipschitz mapping $f\colon\Z^{2}\to\R^{2}$ to a $p(L)$-bilipschitz mapping $F\colon \R^{2}\to \R^{2}$, where $p$ is a polynomial.
	\begin{thm}\label{thm:mainres}
		There is a polynomial function $p\colon \R\to \R$ such that for any $L\geq 1$ and any $L$-bilipschitz mapping $f\colon\Z^{2}\to\R^{2}$ there is a $p(L)$-bilipschitz mapping $F\colon\R^{2}\to\R^{2}$ such that $F|_{\Z^{2}}=f$.
	\end{thm}
	Although above we do not determine the polynomial $p(L)$, we later state a version of the result, in Theorem~\ref{thm:mainres2}, where this polynomial is given explicitly. In our companion article~\cite{DK_ddim} we prove, in all finite dimensions $d\geq 2$, that the bilipschitz extension problem for separated nets of $\R^{d}$ is equivalent to the bilipschitz extension problem for the integer lattice $\Z^{d}$. This allows us to replace $\Z^{2}$ in Theorem~\ref{thm:mainres} with an arbitrary separated net $X$ of $\R^{2}$, gaining a dependence on the separation and net constants of $X$ in the expression bounding the bilipschitz constant of the extension $F$. More precisely, Theorem~\ref{thm:mainres} and \cite[Theorem~1.1]{DK_ddim} admit the following immediate corollary:
	\begin{cor}
		Let the polynomial function $p\colon \R\to\R$ be given by Theorem~\ref{thm:mainres}. Then, for any $L\geq 1$, any $r,R>0$, any $r$-separated $R$-net $X\subseteq \R^{2}$ and any $L$-bilipschitz mapping $f\colon X\to\R^{2}$, there is a bilipschitz mapping $F\colon\R^{2}\to\R^{2}$ such that $F|_{X}=f$ and
		\begin{equation*}
			\bilip(F)\leq Kp\br{50R^{2}K^{3}L},
		\end{equation*}
		where $K:=16\max\set{\frac{6}{r},1}$. 
	\end{cor}

	In what follows, we set out the main extension tools which we develop in the paper and explain how they combine to prove our main result~Theorem~\ref{thm:mainres}. 
	
	One key tool in the proof of Theorem~\ref{thm:mainres} comes from our companion article~\cite{DK_ddim}. Informally, the result says that given a bilipschitz mapping $f\colon\Z^{2}\to \R^{2}$, existence of a bilipschitz extension of $f$ to the whole of $\R^{2}$ is equivalent to the existence of a bilipschitz mapping $G\colon \R^{2}\to \R^{2}$ which behaves appropriately to an extension on a sequence of horizontal lines partitioning $\R^{2}$ into strips. Crucially $G$ itself need not be an extension of $f$: it does not care about the `precise values' of $f(x)$ for $x\in \Z^{2}$, only that these points are correctly partitioned by the $G$-images of the horizontal lines. 
	\begin{thm}\label{thm:bil_ext_equiv}\cite[Thm~1.2 (special case of dimension $d=2$)]{DK_ddim}\newline
		For any $L\geq 1$ and any $L$-bilipschitz mapping $f\colon\Z^{2}\to \R^{2}$, the following are equivalent:
		\begin{enumerate}[(i)]
			\item\label{bil_ext} There exists a bilipschitz mapping $F\colon \R^{2}\to \R^{2}$ such that $F|_{\Z^{2}}=f$.
			\item\label{bil_wsep} There exist $T\in\N$ and a bilipschitz mapping $G\colon \R^{2}\to\R^{2}$ such that the mapping
			\begin{equation*}
				\wt{F}\colon \Z^{2}\cup \br{\R\times\br{T\Z+\frac{1}{2}}}\to \R^{2}, \qquad \wt{F}(x)=\begin{cases}
					f(x) & \text{if }x\in\Z^{2},\\
					G(x) & \text{if }x\in \R\times\br{T\Z+\frac{1}{2}},
				\end{cases}
			\end{equation*}
			is bilipschitz and for every $i\in\Z$ it holds that
			\[
			f\br{\Z^2\cap \br{\R\times\left[T(i-1)+\frac{1}{2},Ti+\frac{1}{2}\right]}}\subseteq G\br{\R\times\left[T(i-1)+\frac{1}{2},Ti+\frac{1}{2}\right]}.
			\]
		\end{enumerate}
		Furthermore, whenever \eqref{bil_wsep} holds with $M_{1}\geq \bilip\br{\wt{F}}$ and $M_{2}\geq \bilip(G)$, the bilipschitz extension $F$ of $f$ in \eqref{bil_ext} may be found with
		\begin{equation*}
			\bilip(F)\leq 
			10^{84}M_{1}^{52}M_{2}^{54}T^{30}.
		\end{equation*}
	\end{thm}
	To briefly explain the role of Theorem~\ref{thm:bil_ext_equiv}, it will be helpful to state one further result from \cite{DK_ddim}:
	\begin{thm}\label{thm:simultaneous_switching}\cite[Theorem~3.1 (special case of dimension $d=2$)]{DK_ddim}\newline
		Let $X\subseteq \R^{2}$. For each $x\in X$ let $y_{x}\in\R^{2}$, $r_{x}$ be a positive real number and $\mf{U}_{x}:=B([x,y_{x}],r_{x})$ (the set of points in $\R^{2}$ with distance less than $r_{x}$ to the line segment $[x,y_{x}]$; see Section~\ref{s:prel}).  Suppose that $(\mf{U}_{x})_{x\in X}$ is pairwise disjoint, 
		\begin{equation*}
			\sup_{x\in X,\, y_{x}\neq x}\frac{r_{x}}{\enorm{y_{x}-x}}\leq \frac{1}{2}\qquad\text{and}\qquad\sup_{x\in X,\, y_{x}\neq x}\frac{\enorm{y_{x}-x}}{r_{x}}<\infty.
		\end{equation*}
		Then there is a bilipschitz mapping $\Gamma\colon \R^{2}\to \R^{2}$ such that
		\begin{enumerate}[(i)]
			\item $\Gamma(x)=y_{x}$ and $\Gamma(y_{x})=x$ for all $x\in X$.
			\item $\Gamma(p)=p$ for all $\displaystyle p\in \R^{2}\setminus \bigcup_{x\in X,\, y_{x}\neq x}\mf{U}_{x}$.
			\item $\displaystyle \bilip(\Gamma)\leq \max\set{1, \sup_{x\in X,\, y_{x}\neq x}\frac{4\enorm{y_{x}-x}^{2}}{r_{x}^{2}}}$.
		\end{enumerate}
	\end{thm}
	Roughly speaking, Theorem~\ref{thm:simultaneous_switching} can be used to turn the `almost extension' $G$ of $f$ in Theorem~\ref{thm:bil_ext_equiv} into a true extension. Composing $G$ with mappings $\Gamma$ given by Theorem~\ref{thm:simultaneous_switching} allows us to `alter the values' of $G$ on a family of specified points. In this way, with significant and delicate additional work involved, we may `correct the values' of $G$ on $\Z^{2}$ so that $G$ becomes an extension of $f$. In particular, $G$ already determines the final extension of $f$ on the horizontal lines $\R\times \set{Ti+\frac{1}{2}}$ for $i\in\Z$.
	
	To understand why Theorem~\ref{thm:simultaneous_switching} holds, we recommend that the reader considers its statement for the case where $X$ is a singleton. Then it just asserts the existence of a bilipschitz mapping which efficiently swaps a specified pair of points $x$ and $y$. Such a bilipschitz mapping can be devised without great difficulty. The full statement of Theorem~\ref{thm:simultaneous_switching} is then obtained by gluing a family of these individual swapping mappings together. In the present work, we call on Theorem~\ref{thm:simultaneous_switching} formally in just one place: in one step of the proof of Theorem~\ref{thm:strips}.
	
	We may now establish our main result, Theorem~\ref{thm:mainres}, by verifying for the given bilipschitz $f\colon \Z^{2}\to \R^{2}$ the condition~\eqref{bil_wsep} of Theorem~\ref{thm:bil_ext_equiv}. This task can be broken down into two steps: First extend $f$ on a sequence of separated horizontal lines partitioning $\R^{2}$ into large strips. That is, extend $f$ to $\Z^{2}\cup\br{\R\times\br{T\Z+\frac{1}{2}}}$ for $T\in\N$ chosen suitably large, to obtain the mapping $\wt{F}$. Next, restrict $\wt{F}$ to the horizontal lines $\R\times\br{T\Z+\frac{1}{2}}$, `throwing away $\Z^{2}$ and $f$', and extend this mapping to the whole of $\R^{2}$ to get $G$. The first step in this strategy, the extension of $f$ to $\wt{F}$, is the most difficult and requires preparation in all subsequent sections of the paper, before it is completed in Section~\ref{s:ext_to_horizont_lines}:
	\begin{restatable*}{thm}{restatestrips}\label{thm:strips}
		Let $L\geq 1$, $T\in\N$, $T\geq 10^{40}L^{16}$ and $f\colon \Z^{2}\to \R^{2}$ be an $L$-bilipschitz mapping. Then there exists a $10^{39}L^{15}$-bilipschitz extension 
		\begin{equation*}
			F\colon \Z^{2}\cup \br{\R\times\br{T\Z+\frac{1}{2}}}\to\R^{2}
		\end{equation*}
		of $f$ such that $\bilip\br{F|_{\R\times\br{T\Z+\frac{1}{2}}}}\leq 10^{24}L^{11}$ and, letting $V_{i}$, for each $i\in\Z$, denote the (unique) connected open subset of $\R^{2}$ with boundary equal to $F\br{\R\times\set{T(i-1)+\frac{1}{2},Ti+\frac{1}{2}}}$, the following statements hold:
		\begin{enumerate}[(i)]
			\item\label{Vi} $V_{i}\cap V_{j}=\emptyset$ whenever $i,j\in\Z$ and $i\neq j$.
			\item\label{images_in_Vi} $\displaystyle f\br{\Z^{2}\cap\br{\R\times\sqbr{T(i-1)+\frac{1}{2},Ti+\frac{1}{2}}}}\subseteq V_{i}$ for each $i\in\Z$.
		\end{enumerate}
	\end{restatable*} 
	The proof of Theorem~\ref{thm:strips} essentially reduces to the task of constructing a single bilipschitz curve as the image of the line $\R\times\set{\frac{1}{2}}$. This same construction can then be repeated on each line of the form $\R\times\br{Ti+\frac{1}{2}}$ with $i\in\Z$, provided the separation $T$ of these lines is taken sufficiently large. The first step in the construction of $F|_{\R\times\set{\frac{1}{2}}}$ is to come up with a bilipschitz curve taking approximately the correct course. This is the objective of Section~\ref{s:shore}: 
	\begin{restatable*}{prop}{initial}\label{prop:initial_red_shore}
		Let $L\geq 1$ and $\gamma\colon \R\to\R^{2}$ be an $L$-Lipschitz mapping such that $\rest{\gamma}{\Z}$ is $L$-bilipschitz. Then there exists a $40L$-bilipschitz mapping $\xi\colon\R\to\R^2$ such that $\infnorm{\gamma-\xi}\leq 8L^5$.
	\end{restatable*}
	Taking $\gamma$ as an $L$-Lipschitz extension of the mapping $\Z\to\R^{2}$, $i\mapsto f(i,0)$, the curve $\xi$ of Proposition~\ref{prop:initial_red_shore} is a crude approximation of $F|_{\R\times\set{\frac{1}{2}}}$ but lacks its most subtle essential property: the image of $F|_{\R\times\set{\frac{1}{2}}}$ must separate $f(\Z\times\Z_{\geq 1})$ from $f(\Z\times\Z_{\leq 0})$. We modify $\xi$ to achieve this. The tool we need for `separating a subset $Y$ of $\Z^{2}$' is developed in Section~\ref{s:separate}:
	\begin{restatable*}{thm}{restateseparation}\label{thm:separation} 
		Let $0<s<1\leq w$, $X\subseteq \R\times[s,w-s]$ be $s$-separated and let $Y\subseteq X$. Then there exists a $\frac{1200w}{s}$-bilipschitz mapping $\gamma=(\gamma_{1},\gamma_{2})\colon \R\to\R^{2}$ such that the following statements hold:
		\begin{enumerate}[(i)]
			\item\label{playground} $\abs{\gamma_{1}(t)-t}\leq \frac{s}{4}$ and $0\leq \gamma_{2}(t)\leq w$ for all $t\in\R$.
			\item\label{stayaway} $\dist(x,\gamma(t))\geq \frac{2s}{100}$ for all $t\in\R$ and $x\in X$.	
			\item\label{capture_avoid} Letting $W_{+}$ and $W_{-}$ denote the connected components of $\R^{2}\setminus \gamma(\R)$ containing $\R\times (w,\infty)$ and $\R\times (-\infty,0)$ respectively, we have $Y\subseteq W_{-}$ and $X\setminus Y\subseteq W_{+}$.		
		\end{enumerate}
	\end{restatable*}
	Let us return to the setting of Theorem~\ref{thm:strips} and imagine that the image points $f(\Z\times\Z_{\geq 1})$ are coloured red, whilst the image points $f(\Z\times \Z_{\leq 0})$ are coloured green. Suppose further that $\xi$ from Proposition~\ref{prop:initial_red_shore} provides an approximation of our extension $F$ to be constructed on $\R\times\set{\frac{1}{2}}$. If we stand on the curve $\xi$ and look out to either side, then far out to horizon on one side we will see a `red sea' of points and far out to the horizon on the other side a `green sea'. Only relatively nearby to the curve $\xi$ will we see a mixture of both red and green points. Theorem~\ref{thm:separation} gives the means to separate `red' and `green' if these points occur only inside a horizontal strip. However, our curve $\xi$ may be very winding. To overcome this difficulty, we apply a bilipschitz mapping to transform the picture, so that this `twisting band of uncertainty' around $\xi$, where red and green points mix, becomes a horizontal strip. Such a bilipschitz mapping exists, since, viewing $\R$ as a subset of $\R^{2}$, $\xi\colon \R\to\R^{2}$ may be extended to a bilipschitz mapping of $\R^{2}$.
	
	With Theorem~\ref{thm:strips} providing a suitable way to define the mapping $\wt{F}$ from Theorem~\ref{thm:bil_ext_equiv}\eqref{bil_wsep} on the set $\R\times\br{T\Z+\frac{1}{2}}$, the only thing that remains to establish \eqref{bil_wsep} is to extend the restriction of $\wt{F}$ on this set to the whole of $\R^{2}$ (to get $G$). At this point we may forget about the original given mapping $f\colon \Z^{2}\to \R^{2}$; any bilipschitz extension of $\wt{F}|_{\R\times\br{T\Z+\frac{1}{2}}}$ will suffice! Since the restricted $\wt{F}$ is defined on a sequence of separated horizontal lines partitioning $\R^{2}$ into strips, we need to extend on each strip. The bilipschitz extension problem for the boundary of a strip is resolved in Section~\ref{ext_to_strip}:
	\begin{restatable*}{thm}{restateexttostrip}\label{thm:ext_to_strip} 
		Let $L\geq 1$, $h>0$ and $f\colon \R\times\set{0,h}\to\R^{2}$ be an $L$-bilipschitz mapping. Then there is a $10^{19}L^{7}$-bilipschitz mapping $F\colon \R\times\sqbr{0,h}\to \R^{2}$ extending $f$.
	\end{restatable*}
	Together, Theorem~\ref{thm:strips} and Theorem~\ref{thm:ext_to_strip}, allow for the verification of Theorem~\ref{thm:bil_ext_equiv}\eqref{bil_wsep}, giving the proof of our main result Theorem~\ref{thm:mainres}. Before giving the formal proof of Theorem~\ref{thm:mainres}, we restate it for the reader's convenience. This time, we actually state a stronger version of Theorem~\ref{thm:mainres} which expresses the polynomial bound on the bilipschitz constant of the extension explicitly. 
	\begin{thm}[Stronger version of Theorem~\ref{thm:mainres}]\label{thm:mainres2}
		Let $L\geq 1$ and $f\colon \Z^{2}\to \R^{2}$ be an $L$-bilipschitz mapping. Then there is a bilipschitz mapping $F\colon \R^{2}\to\R^{2}$ such that $F|_{\Z^{2}}=f$ and
		\begin{equation*}
			\bilip(F)\leq 10^{20000}L^{6000}.
		\end{equation*}
	\end{thm}
	\begin{proof}[Proof of Theorems~\ref{thm:mainres} and~\ref{thm:mainres2}]
		Theorem~\ref{thm:mainres2} is a more detailed version of Theorem~\ref{thm:mainres}, so we only need to prove~\ref{thm:mainres2}. Let 
		\begin{multline*}
			T:=\ceil{10^{40}L^{16}},\qquad M_{1}:=10^{39}L^{15},\qquad K:=10^{24}L^{11},\quad \text{and the}\\
			M_{1}\text{-bilipschitz mapping }\wt{F}\colon \Z^{2}\cup\br{\R\times\br{T\Z+\frac{1}{2}}}\to\R^{2}\text{ with } \bilip\br{\wt{F}|_{\R\times\br{T\Z+\frac{1}{2}}}}\leq K,
		\end{multline*}
		be given by Theorem~\ref{thm:strips}. Next, for each $i\in\Z$, let $F_{i}\colon \R\times\sqbr{T(i-1)+\frac{1}{2},Ti+\frac{1}{2}}\to \R^{2}$ be the $10^{19}K^{7}$-bilipschitz extension of $\wt{F}|_{\R\times\set{T(i-1)+\frac{1}{2},Ti+\frac{1}{2}}}$ given by Theorem~\ref{thm:ext_to_strip} (appropriately shifted). We define $G\colon \R^{2}\to \R^{2}$ as the `gluing together' of the family of mappings $(F_{i})_{i\in\Z}$. More precisely, in the notation defined in Definition~\ref{def:gluing}, we set $G:=\bigcup_{i\in\Z}F_{i}$. It is an exercise to verify, using induction and Corollary~\ref{cor:gluing_bilip}, that $G$ is bilipschitz with
		\begin{equation*}
			\bilip(G)\leq \sup_{i\in\Z}\bilip(F_{i})\leq 10^{19}K^{7}\leq 10^{200}L^{77}=:M_{2}.
		\end{equation*}
		As a homeomorphism $\R^{2}\to\R^{2}$, $G$ has to map, for each $i\in\Z$, the connected open set with boundary equal to $\R\times\set{T(i-1)+\frac{1}{2},Ti+\frac{1}{2}}$ onto the connected open set with boundary equal to $G\br{\R\times\set{T(i-1)+\frac{1}{2},Ti+\frac{1}{2}}}=\wt{F}\br{\R\times\set{T(i-1)+\frac{1}{2},Ti+\frac{1}{2}}}$. In other words we have
		\begin{equation*}
			G\br{\R\times\br{T(i-1)+\frac{1}{2},Ti+\frac{1}{2}}}=V_{i} \qquad \text{for all $i\in\Z$,}
		\end{equation*}
		where $V_{i}$ for $i\in\Z$ is defined in the statement of Theorem~\ref{thm:strips}. The proof is now completed by Theorem~\ref{thm:bil_ext_equiv}: observe that $T$, $G$ verify statement \eqref{bil_wsep} of Theorem~\ref{thm:bil_ext_equiv} and $M_{1}$, $M_{2}$, the hypothesis of its `Furthermore' statement. To refine the upper bound on $\bilip(F)$ from Theorem~\ref{thm:bil_ext_equiv} to the form in the present statement, we apply $T\leq 10^{41}L^{16}$.
	\end{proof}

\section{Preliminaries.}\label{s:prel}
\subsection*{Basic notation and facts.}
We use largely a common notation in the present article and its companion article~\cite{DK_ddim}. We write $:=$ to signify a definition by equality. For $n\in\N$ the notation $[n]$ refers to $\set{1,\ldots, n}$. Moreover, for $P\in\R$, the notation $\Z_{\geq P}$ (e.g.) stands for $\set{i\in\Z\colon i\geq P}$. We write $e_{1},e_{2}$ for the standard basis of $\R^{2}$. The orthogonal projection to the $i$-th coordinate in $\R^2$ is denoted by $\proj_i$ for $i=1,2$.

Given two sets $A,A'\subseteq \R^{2}$ we let 
\[
\dist(A,A'):=\inf\set{\enorm{a-a'}\colon a\in A,\,a'\in A'}.
\]
In the case that $A=\set{a}$ is a singleton we just write $\dist(a,A')$ instead of $\dist(\set{a},A')$.
The expression $\diam(A)$ stands for the diameter of the set $A$.

Given $A\subset \R^2$ and $r>0$, we say that $A$ is \emph{$r$-separated} if $\enorm{a-a'}\geq r$ for every $a,a'\in A, a\neq a'$. We say that $A$ is an \emph{$r$-net} if $\dist(x,A)\leq r$ for every $x\in \R^2$. If there are $r,s>0$ such that $A$ is an $s$-separated $r$-net, we call $A$ a \emph{separated net}. 

We write $B(x,r)$ and $\cl{B}(x,r)$ respectively for the open and closed \emph{euclidean} balls with centre $x\in\R^{d}$ and radius $r\geq 0$. Moreover, we use the same notation for neighbourhoods of sets, i.e, $B(A,r):=\bigcup_{x\in A}B(x,r)$, where $A\subseteq\R^2$, and similarly for $\cl{B}(A,r)$. The letter $B$ is reserved for balls generally. Occasionally, we use also balls with respect to the $\ell_{\infty}$-norm; these are denoted by $B_\infty$ (or $\cl{B}_\infty$ for the closed version).

For $x,y\in\R^2$, by $[x,y]$ we mean the closed line segment with endpoints $x$ and $y$.

Given $A\subseteq\R^2$, by $\partial A, \inter A$ and $\cl{A}$, we denote the boundary, interior and closure of $A$ in $\R^2$, respectively. For a mapping $f\colon A'\subseteq\R^2\to\R^2$ such that $A\subseteq A'$, we write $\rest{f}{A}$ for the restriction of $f$ to $A$. Given a second mapping $g\colon A'\to\R^2$, we write $f\equiv g$ if $f(x)=g(x)$ for every $x\in A'$. The notation $\dom(f), \img(f)$ refer to the domain and the image of $f$, respectively. We say that $f$ is a \emph{homeomorphism} if $f$ is injective, continuous and that its inverse mapping $f^{-1}\colon f(A')\to \R^{2}$ is also continuous (sometimes the term \emph{embedding} is used instead in the literature).

\subsection*{Lipschitz and bilipschitz mappings.}
For a subset $A$ of $\R$ or $\R^{2}$ and a mapping $f\colon A\to\R^{2}$ we let
\begin{linenomath}
	\begin{equation*}
		\lip(f):=\sup\left\{\frac{\enorm{f(y)-f(x)}}{\enorm{y-x}}\colon x,y\in A, x\neq y\right\}.
	\end{equation*}
\end{linenomath}
In the case that $f$ is injective, we further define 
\begin{linenomath}
	\begin{equation*}
		\bilip(f):=\max\set{\lip(f),\lip(f^{-1})}.
	\end{equation*}
\end{linenomath}
Given $L\geq 1$, we say that $f$ is $L$\emph{-Lipschitz} if $\lip(f)\leq L$ and that it is \emph{$L$-bilipschitz} if $\bilip(f)\leq L$. Further, we say that $f$ is \emph{Lipschitz} (\emph{bilipschitz}) if there is $L<\infty$ such that $f$ is $L$-Lipschitz ($L$-bilipschitz).

Throughout the paper, we will construct many polygonal curves as images of piecewise linear mappings defined on $\R$. The following criterion, providing sufficient conditions for such a mapping to be bilipschitz, is going to be used several times.
\begin{lemma}\label{lemma:pwaff_bil}
	Let $\gamma\colon\R\to \R^{2}$ be a mapping, $(p_{n})_{n\in\Z}$ be a strictly increasing double-sided sequence of real numbers and $\alpha\in(0,\pi)$ be such that for every $n\in\Z$ the restriction $\gamma|_{[p_{n},p_{n+1}]}$ is affine and the angle at $\gamma(p_{n})$ in the (possibly degenerate) triangle with vertices $\gamma(p_{n-1})$, $\gamma(p_{n})$ and $\gamma(p_{n+1})$ is at least $\alpha$. Let $C_{1},C_{2},C_{3},C_{4}>0$ be constants satisfying the following conditions:
	\begin{enumerate}[(a)]
		\item\label{speed} $\frac{1}{C_{1}}\leq \enorm{\gamma'(t)}\leq C_{1}$ for almost every $t\in \R$.
		\item\label{buffer} For any $n,m\in \Z$ with $\abs{n-m}\geq 2$ we have that 
		\begin{equation*}
			\dist\br{\sqbr{\gamma(p_{n-1}),\gamma(p_{n})},\sqbr{\gamma(p_{m-1}),\gamma(p_{m})}}\geq C_{2}.
		\end{equation*}
		\item 
		\begin{enumerate}[(i)]
			\item\label{far_away} For any $t_{1},t_{2}\in\R$ with $\abs{t_{2}-t_{1}}\geq C_{3}$ we have $\enorm{\gamma(t_{2})-\gamma(t_{1})}\geq \frac{1}{C_{4}}\abs{t_{2}-t_{1}}$, \textbf{or}
			\item\label{gamma_on_nodes} $\gamma|_{\set{p_{n}\colon n\in\Z}}$ is $C_{1}$-bilipschitz and $\enorm{\gamma(p_{n})-\gamma(p_{n-1})}=C_{3}$ for all $n\in\Z$.
		\end{enumerate}	
	\end{enumerate}
	Then $\gamma$ is $\max\set{\frac{4C_{1}}{\sqrt{1-\cos\alpha}},C_{4},\frac{C_{3}}{C_{2}}}$-bilipschitz if \eqref{far_away} holds and $\gamma$ is $\max\set{\frac{4C_{1}}{\sqrt{1-\cos \alpha}},\frac{10C_{1}C_{3}}{C_{2}}}$-bilipschitz if \eqref{gamma_on_nodes} holds.
\end{lemma}
\begin{proof}
	The fact that $\gamma$ is piecewise affine and the second inequality in \eqref{speed} imply that $\gamma$ is $C_{1}$-Lipschitz. It remains to check the lower bound in the bilipschitz condition. Let $t_{1},t_{2}\in\R$ with $t_{1}<t_{2}$ and let $n_{1}\leq n_{2}$ be the indices for which $t_{i}\in [p_{n_{i}-1},p_{n_{i}}]$. We will bound the quotient $\frac{\enorm{\gamma(t_{2})-\gamma(t_{1})}}{\abs{t_{2}-t_{1}}}$ below by a constant depending only on $C_{1},\ldots,C_{4}$ and $\alpha$. We distinguish three cases: $n_{1}=n_{2}$, $n_{2}=n_{1}+1$ and $n_{2}\geq n_{1}+2$. In the first case, $n_{1}=n_{2}$, we may apply the first inequality of \eqref{speed} to deduce $\enorm{\gamma(t_{2})-\gamma(t_{1})}\geq \frac{1}{C_{1}}\abs{t_{2}-t_{1}}$. In the second case, $n_{2}=n_{1}+1$ we have, using \eqref{speed} again, that 
	\begin{equation*}
		\frac{1}{C_{1}}\abs{t_{i}-p_{n_{1}}}\leq \enorm{\gamma(t_{i})-\gamma(p_{n_{1}})}, \qquad \text{for }i=1,2.
	\end{equation*}
	Hence $\enorm{\gamma(t_{1})-\gamma(p_{n_{1}})}+\enorm{\gamma(t_{2})-\gamma(p_{n_{1}})}\geq \frac{1}{C_{1}}\br{\abs{t_{1}-p_{n_{1}}}+\abs{t_{2}-p_{n_{1}}}}=\frac{1}{C_{1}}\abs{t_{2}-t_{1}}$. Without loss of generality we may assume that $\enorm{\gamma(t_{2})-\gamma(p_{n_{1}})}\geq \enorm{\gamma(t_{1})-\gamma(p_{n_{1}})}$ so that $\enorm{\gamma(t_{2})-\gamma(p_{n_{1}})}\geq \frac{1}{2C_{1}}\abs{t_{2}-t_{1}}$. If $\enorm{\gamma(t_{1})-\gamma(p_{n_{1}})}\leq \frac{1}{2}\enorm{\gamma(t_{2})-\gamma(p_{n_{1}})}$ we get 
	\begin{multline*}
		\enorm{\gamma(t_{2})-\gamma(t_{1})}\geq \enorm{\gamma(t_{2})-\gamma(p_{n_{1}})}-\enorm{\gamma(t_{1})-\gamma(p_{n_{1}})}\\
		\geq \frac{1}{2}\enorm{\gamma(t_{2})-\gamma(p_{n_{1}})}\geq \frac{1}{4C_{1}}\abs{t_{2}-t_{1}}.
	\end{multline*}
	Otherwise, $\frac{1}{4C_{1}}\abs{t_{2}-t_{1}}\leq\frac{1}{2}\enorm{\gamma(t_{2})-\gamma(p_{n_{1}})}\leq\enorm{\gamma(t_{1})-\gamma(p_{n_{1}})}\leq \enorm{\gamma(t_{2})-\gamma(p_{n_{1}})}$ and $\gamma(t_{1})$, $\gamma(p_{n_{1}})$ and $\gamma(t_{2})$ are vertices of a triangle where the two sides meeting at $\gamma(p_{n_{1}})$ have lengths at least $\frac{1}{4C_{1}}\abs{t_{2}-t_{1}}$ and the angle at which they meet is at least $\alpha$, by hypothesis. This information implies a lower bound on the length of the side opposite $\gamma(p_{n_{1}})$, namely
	\begin{equation*}
		\enorm{\gamma(t_{2})-\gamma(t_{1})}^{2}\geq 2\enorm{\gamma(t_{1})-\gamma(p_{n_{1}})}\enorm{\gamma(t_{2})-\gamma(p_{n_{1}})}(1-\cos \alpha)\geq \frac{1-\cos \alpha}{8C_{1}^{2}}\abs{t_{2}-t_{1}}^{2}.
	\end{equation*}

	In the last remaining case, $n_{2}\geq n_{1}+2$, we will obtain different bounds depending on whether we assume \eqref{far_away} or \eqref{gamma_on_nodes}. First, we assume \eqref{far_away}. This provides the lower bound of $1/C_{4}$ for the specified quotient when $\abs{t_{2}-t_{1}}\geq C_{3}$. On the other hand, if $\abs{t_{2}-t_{1}}\leq C_{3}$, we may apply \eqref{buffer} to get
	\begin{equation*}
		\frac{\enorm{\gamma(t_{2})-\gamma(t_{1})}}{\abs{t_{2}-t_{1}}}\geq \frac{C_{2}}{C_{3}}.
	\end{equation*}
	Now suppose that instead of \eqref{far_away} we have \eqref{gamma_on_nodes}. We will make use of the following bounds:
	\begin{align*}
		\abs{t_{1}-p_{n_{1}}}&\leq C_{1}\abs{\gamma(t_{1})-\gamma(p_{n_{1}})}\leq  C_{1}\abs{\gamma(p_{n_{1}})-\gamma(p_{n_{1}-1})}=C_{1}C_{3},\qquad\text{and}\\
		\abs{t_{2}-p_{n_{2}-1}}&\leq C_{1}C_{3}\quad\text{(similarly).} 
	\end{align*}
	We distinguish two cases: 
	If $\abs{p_{n_{1}}-p_{n_{2}-1}}\geq 8C_{1}C_{3}$, then $\abs{t_{2}-t_{1}}\geq 8C_{1}C_{3}$ and we deduce 
	\begin{multline*}
		\enorm{\gamma(t_{2})-\gamma(t_{1})}\geq \enorm{\gamma(p_{n_{2}-1})-\gamma(p_{n_{1}})}-2C_{3}\geq \frac{\abs{p_{n_{2}-1}-p_{n_{1}}}}{C_{1}}-2C_{3}\\
		\geq \frac{\abs{t_{2}-t_{1}}}{C_{1}}-4C_{3}\geq \frac{\abs{t_{2}-t_{1}}}{2C_{1}}.
	\end{multline*}
	If $\abs{p_{n_{1}}-p_{n_{2}-1}}< 8C_{1}C_{3}$, then $\abs{t_{2}-t_{1}}\leq 10C_{1}C_{3}$ and, from \eqref{buffer}, we get
	\begin{equation*}
		\enorm{\gamma(t_{2})-\gamma(t_{1})}\geq C_{2}\geq \frac{C_{2}}{10C_{1}C_{3}}\abs{t_{2}-t_{1}}.
	\end{equation*}
	
	It remains to note that all the lower bounds we obtained for $\frac{\enorm{\gamma(t_{2})-\gamma(t_{1})}}{\abs{t_{2}-t_{1}}}$ over all cases are at least $\br{\max\set{\frac{4C_{1}}{\sqrt{1-\cos\alpha}},C_{4},\frac{C_{3}}{C_{2}}}}^{-1}$, assuming \eqref{far_away}, and at least, $\br{\max\set{\frac{4C_{1}}{\sqrt{1-\cos \alpha}},\frac{10C_{1}C_{3}}{C_{2}}}}^{-1}$, assuming \eqref{gamma_on_nodes} instead of \eqref{far_away}.
\end{proof}

\paragraph{Extensions.}
Kovalev's bilipschitz extension results~\cite{Kovalev_bilip_ext_from_line,Kovalev_optimal_bilip}, from the circle to the disc and from the line to the plane, will have an important role in the present paper. Of course there is no meaningful difference between bilipschitz extension from the circle and from the square; the next lemma is a bridge between them.

\begin{lemma}\label{lemma:circle_to_square}
	There is a $10$-bilipschitz bijection between $\cl{B}(0,1)\subseteq \R^{2}$ and $[0,1]^{2}$.
\end{lemma}
\begin{proof}
	We will show that the mapping $\Phi\colon\R^{2}\to\R^{2}$ defined by $\Phi(x)=\begin{cases}
		\frac{\enorm{x}}{\infnorm{x}}x & \text{if }x\neq 0,\\
		0 & \text{if }x=0
	\end{cases}$ is $2(1+\sqrt{2})$-bilipschitz. Since $\Phi|_{\cl{B}(0,1)}$ is a bijection from $\cl{B}(0,1)$ to $[-1,1]^{2}$ and the latter set is clearly equivalent via a $2$-bilipschitz bijection to $[0,1]^{2}$, the result then follows.
	
	We first show that $\Phi$ is $2(1+\sqrt{2})$-Lipschitz. Since $\Phi$ is continuous, it suffices to derive, for $x,y\in\R^{2}\setminus\set{0}$, the following
	\begin{multline*}
		\enorm{\Phi(y)-\Phi(x)}=\enorm{\frac{\enorm{y}}{\infnorm{y}}(y-x)+x\enorm{y}\br{\frac{1}{\infnorm{y}}-\frac{1}{\infnorm{x}}}+\frac{x}{\infnorm{x}}\br{\enorm{y}-\enorm{x}}}\\
		\leq \frac{\enorm{y}}{\infnorm{y}}\enorm{y-x}+\frac{\enorm{x}\enorm{y}}{\infnorm{x}\infnorm{y}}\abs{\infnorm{x}-\infnorm{y}}+\frac{\enorm{x}}{\infnorm{x}}\abs{\enorm{y}-\enorm{x}}\\
		\leq \br{\sqrt{2}+2+\sqrt{2}}\enorm{y-x}.
	\end{multline*}
	Since $\Phi^{-1}(x)=\begin{cases}
		\frac{\infnorm{x}}{\enorm{x}}x & \text{if }x\neq 0,\\
		0 & \text{if }x=0
	\end{cases}$, a symmetric argument, interchanging the norms $\enorm{\cdot}$ and $\infnorm{\cdot}$, shows that $\lip(\Phi^{-1})\leq 2\sqrt{2}+1$.
\end{proof}
It will be convenient for us to work with the next restatement of Kovalev's bilipschitz extension~\cite[Thm.~1.1]{Kovalev_optimal_bilip}, formulated for the square rather than the circle using Lemma~\ref{lemma:circle_to_square}.
\begin{thm}[Kovalev~{\cite[Thm.~1.1]{Kovalev_optimal_bilip}}]\label{thm:Kovalev}
	Let $f\colon \partial[0, 1]^2\to\R^2$ be an $L$-bilipschitz mapping. Then there is a $2\cdot 10^{16}L$-bilipschitz mapping $F\colon[0, 1]^2\to\R^2$ extending $f$.
\end{thm}
\begin{proof}
	Let $\Phi\colon \cl{B}(0,1)\to [0,1]^{2}$ be a $10$-bilipschitz bijection given by Lemma~\ref{lemma:circle_to_square}. Then, \cite[Thm~1.1]{Kovalev_optimal_bilip} provides a $2\cdot 10^{14}\cdot10L$-bilipschitz extension $G\colon \cl{B}(0,1)\to \R^{2}$ of the $10L$-bilipschitz mapping $f\circ \Phi|_{S^{1}}$. The desired extension of $f$ is now given by $G\circ \Phi^{-1}$.
\end{proof}
Similarly, we formulate Kovalev's bilipschitz extension from the line to the plane in the form most convenient for our use. We also add some extra details about connected components to the statement.
\begin{thm}[Kovalev~{\cite[Thm~1.2]{Kovalev_bilip_ext_from_line}}]\label{lemma:normalise_step}
	Let $L\geq1$, $\xi\colon \R\to \R^{2}$ be an $L$-bilipschitz mapping and $W_{-}$, $W_{+}$ denote the two connected components of $\R^{2}\setminus \xi(\R)$. Then there exists a $2000L$-bilipschitz mapping $\Phi\colon \R^{2}\to \R^{2}$ such that $\Phi\circ \xi(t)=(t,0)$ for all $t\in\R$,
	\begin{equation*}
		\Phi\br{W_{-}}=\R\times (-\infty,0) \qquad \text{and}\qquad \Phi\br{W_{+}}=\R\times (0,\infty).
	\end{equation*}
\end{thm}
\begin{proof}
	By \cite[Thm~1.2]{Kovalev_bilip_ext_from_line} there exists a $2000L$-bilipschitz mapping $\Delta\colon \R^{2}\to\R^{2}$ such that $\Delta(t,0)=\xi(t)$ for all $t\in\R$. Then $\Delta^{-1}\circ \xi(t)=(t,0)$ for all $t\in\R$. Moreover, as a homeomorphism of $\R^{2}$, $\Delta^{-1}$ maps the two connected components $W_{-}$ and $W_{+}$ of $\R^{2}\setminus \xi(\R)$ onto the two connected components of $\R^{2}\setminus \Delta^{-1}\circ \xi(\R)=\R^{2}\setminus \br{\R\times \set{0}}$. Accordingly we have that $\Delta^{-1}(W_{+})$ is either equal to $\R\times (0,\infty)$ or to $\R\times (-\infty,0)$. In the former case we take $\Phi=\Delta^{-1}$ and in the latter case we take $\Phi=R_{0}\circ \Delta^{-1}$ where $R_{0}\colon\R^{2}\to\R^{2}$ denotes the reflection in the line $\R\times\set{0}$.
\end{proof}

\paragraph{Gluing of maps.}
We will often build a mapping on a set $A\subseteq\R^2$ from two or more mappings defined on subsets of $A$. Assuming the latter are Lipschitz or bilipschitz, we usually want then to provide guarantees of the same type for the new mapping.

We refer to the operation of putting several mappings together as \emph{gluing}. We introduce this terminology formally below together with a special notation for it:
\begin{define}[Gluing of mappings.]\label{def:gluing}
	Given two mappings $f_i\colon A_i\subset\R^2\to\R^2, i\in\set{1, 2},$ we say that $f_1\cup f_2$ is \emph{well-defined} if $\rest{f_1}{A_1\cap A_2}\equiv\rest{f_2}{A_1\cap A_2}$.
	In this case $f_1\cup f_2$ denotes the mapping
	\[f_1\cup f_2\colon A_1\cup A_2\to\R^2\]
	defined as $(f_1\cup f_2)(x):=f_i(x)$ whenever $x\in A_i$ for $i\in\set{1, 2}$.
	
	More generally, if $(f_i)_{i\in\mc{I}}$ is a collection of mappings $f_i\colon A_i\subset\R^2\to\R^2, i\in\mc{I}$, such that for every $i, j\in \mc{I}, i\neq j$ the mapping $f_i\cup f_j$ is well-defined, then the gluing operation yields also a well-defined mapping 
	\[\bigcup_{i\in\mc{I}}f_i\colon\bigcup_{i\in\mc{I}} A_i\to\R^2\]
	setting $\br{\bigcup_{i\in\mc{I}}f_i}(x):=f_i(x)$ whenever $x\in A_i$ for $i\in\mc{I}$.
\end{define}

The following easy observation about gluing of Lipschitz mappings is a restatement of \cite[Lem.~2.22]{Elems_of_Lip_topology}.
\begin{lemma}\label{lem:gluing_lip}
	Let $f_i\colon A_i\subset\R^2\to\R^2$, $i\in\set{1, 2}$, be two $L$-Lipschitz mappings such that $f_1\cup f_2$ is well-defined. If there is $C>0$ such that for every $x\in A_1\setminus A_2, y\in A_2\setminus A_1$ there is $z(x,y)\in A_1\cap A_2$ satisfying $C\enorm{x-y}\geq\enorm{x-z(x,y)}+\enorm{y-z(x,y)}$, then $f_1\cup f_2$ is $CL$-Lipschitz.
\end{lemma}
\begin{proof}
	For each $x\in A_{1}\setminus A_{2}$ and $y\in A_{2}\setminus A_{1}$ 
	\begin{align*}
		\enorm{(f_1\cup f_2)(x)-(f_1\cup f_2)(y)}&=\enorm{f_1(x)-f_1(z(x,y))+f_2(z(x,y))-f_2(y)}\\
		&\leq L\enorm{x-z(x,y)}+L\enorm{z(x,y)-y}\leq CL\enorm{x-y}.\qedhere
	\end{align*}
\end{proof}

We will use the following consequence of Lemma~\ref{lem:gluing_lip} to glue bilipschitz mappings in a situation where there is an obvious choice of $z(\cdot, \cdot)$ making $C=1$.
\begin{cor}\label{cor:gluing_bilip}
	Let $f_i\colon A_i\subset\R^2\to\R^2$, $i\in\set{1, 2}$, be two $L$-bilipschitz mappings such that $f_1\cup f_2$ is well-defined. If $\partial A_1\subseteq A_1\cap A_2$ and $f_1(A_1)\cap f_2(A_2)=f_1(A_1\cap A_2)$, then $f_1\cup f_2$ is $L$-bilipschitz.
\end{cor}
\begin{proof}
	We set $f:=f_1\cup f_2$.
	The assumptions ensure that $f^{-1}$ is well-defined and that $A_1$ is closed. Since $f_{1}$ is bilipschitz, $f_{1}(A_{1})$ is also closed, so $\partial f_{1}(A_{1})\subseteq f_{1}(A_{1})$. Using Brouwer's Invariance of Domain~\cite[Thm~2B.3]{Hatcher02}, we conclude that $\partial f_{1}(A_{1})\subseteq f_{1}(\partial A_{1})$. Hence $\partial f_{1}(A_{1})\subseteq f_{1}(A_{1}\cap A_{2})=f_{1}(A_{1})\cap f_{2}(A_{2})$. Thus, the assumptions are symmetric with respect to $(f_{1},f_{2},f)$ and $(f_{1}^{-1},f_{2}^{-1},f^{-1})$ and it suffices to show that $f$ is $L$-Lipschitz, as follows: Given any $x_1\in A_1, x_2\in A_2\setminus A_1$ we take $z(x_1, x_2)$ as any point of $[x_1, x_2]\cap \partial A_1$ and invoke Lemma~\ref{lem:gluing_lip} with $C=1$.
\end{proof}

\section{First approximation of the shore.}\label{s:shore}
To place the present section in context, recall that the overarching aim of the paper is to construct, for a given bilipschitz mapping $f\colon \Z^{2}\to \R^{2}$, a bilipschitz extension $F\colon \R^{2}\to \R^{2}$. The objective of this section is to construct what may be thought of as a first approximation of the extension on a horizontal line, more precisely, of the mapping $F|_{\R\times \set{i+\frac{1}{2}}}$, for some fixed integer $i$. The latter can be thought of as a bilipschitz curve loosely following the points $f(x,i)$ for $x\in\Z$ in the correct order. To get such a curve, we may start with a Lipschitz curve $\gamma$ going from $f(x,i)$ to $f(x,i+1)$ along straight line segments and then modify it in steps. In the present section, we show that such a Lipschitz curve $\gamma$ admits a relatively close by bilipschitz curve, which will later be modified further to become $F|_{\R\times\set{i+\frac{1}{2}}}$. The title of the present section comes from our imagination of the set $f(\Z\times \Z_{\geq i+1})$ as a sea and the curve $F|_{\R\times \set{i+\frac{1}{2}}}$ as defining its coastline.
\initial
In the proof of Proposition~\ref{prop:initial_red_shore}, we will use a simple geometric fact:
\begin{lemma}\label{lemma:twodiscs}
	Let $u,v\in\R^{2}$ be such that $\enorm{v-u}=1$. Then
	\begin{equation*}
		B\br{[u,v],\frac{1}{\sqrt{2}}}\subseteq B\br{u,\frac{\sqrt{3}}{2}}\cup B\br{v,\frac{\sqrt{3}}{2}}.
	\end{equation*}
\end{lemma}
\begin{proof}(elementary exercise) Let $x\in B\br{[u,v],\frac{1}{\sqrt{2}}}$. Our task is to show that $x\in B\br{u,\frac{\sqrt{3}}{2}}\cup B\br{v,\frac{\sqrt{3}}{2}}$. We may assume that $\dist(x,[u,v])$ is attained at $0\in [u,v]\setminus\set{u,v}$, making $x$ and $u$ (and $x$ and $v$) orthogonal. We also assume, without loss of generality that $\enorm{u}\leq 1/2$ (otherwise interchange $u$ and $v$). Then $\enorm{x-u}^{2}=\enorm{u}^{2}+\enorm{x}^{2}< \frac{1}{4}+\frac{1}{2}$.
\end{proof}

\begin{proof}[Proof of Proposition~\ref{prop:initial_red_shore}]
	We fix a polynomial $\omega(L)$, whose precise value will be specified later in the proof. We aim to construct an increasing double-sided sequence of points $\br{p_i}_{i\in\Z}$ in $\R$ so that $\enorm{\gamma(p_{i})-\gamma(p_{i-1})}=\omega(L)$ for all $i\in\Z$ and with additional properties to be specified later. Next, we define $\xi(p_i):=\gamma(p_i)$ for every $i\in\Z$ and then extend $\xi$ to each interval $[p_i, p_{i+1}]$ affinely.
	
	\begin{claim}\label{cl:coarse_bilip}
		Let $s,t\in\R$ satisfy $\abs{s-t}\geq 4L^2$. Then $\enorm{\gamma(s)-\gamma(t)}\geq\frac{\abs{s-t}}{2L}$.
	\end{claim}
	\begin{proof}
		Without loss of generality, we assume that $s>t$. We get 
		\[
		\enorm{\gamma(s)-\gamma(t)}\geq \enorm{\gamma(\ceil{s})-\gamma(\floor{t})}-2L\geq\frac{\abs{\ceil{s}-\floor{t}}}{L}-2L	\geq \frac{\abs{s-t}}{L}-2L\geq \frac{\abs{s-t}}{2L}. \qedhere
		\]
	\end{proof}

	Consider the set 
	\[
	\mc{P}:=\set{(t_-, t_+)\in (-\infty,0]\times[0, \infty)\colon \enorm{\gamma(t_-)-\gamma(t_+)}\leq \omega(L)}.
	\]
	By continuity of $\gamma$, we observe that $\mc{P}$ is closed. Moreover, by Claim~\ref{cl:coarse_bilip}, it is also bounded, and thus, it is compact. Consequently, $p_{0}:=\min\proj_1(\mc{P})$ is well-defined. This definition of $p_{0}$ together with $\lim_{t\to \infty}\enorm{\gamma(t)}=\infty$, which follows from Claim~\ref{cl:coarse_bilip}, implies
	\begin{equation}\label{eq:sep_plus_minus}
		\dist(\gamma(t), \gamma([0, \infty)))> \omega(L)\qquad \text{for every $t<p_{0}$.}
	\end{equation}
	Next, we set $p_i:=\sup\set{t\in\R\colon\enorm{\gamma(t)-\gamma(p_{i-1})}\leq \omega(L)}$ assuming $i>0$ and that $p_{i-1}$ is already defined.

	Similarly, we set $p_i:=\inf\set{t\in\R\colon \enorm{\gamma(t)-\gamma(p_{i+1})}\leq \omega(L)}$  if $i<0$ and $p_{i+1}$ is already defined. It follows that $p_{i}>p_{i-1}$ and $\enorm{\gamma(p_{i})-\gamma(p_{i-1})}=\omega(L)$ for all $i\in\Z$, implying the first inequality below:
	\begin{equation}\label{eq:bi_step}
		\frac{\omega(L)}{L}\leq p_{i}-p_{i-1}\leq 2L\omega(L) \qquad \text{for all $i\in\Z$.}
	\end{equation}
	The second inequality is given by Claim~\ref{cl:coarse_bilip} when we impose the condition
	\begin{equation}\label{eq:pL2}
		\omega(L)\geq 4L^3
	\end{equation}
	on $\omega(L)$. We also point out that since $(p_i)_{i\in\Z}$ is increasing and $p_1\in\proj_2(\mc{P})$, we have that $p_i\geq 0$ for every $i>0$. 
	
	We now define $\xi\colon\R\to\R^{2}$ by prescribing that $\xi(p_i)=\gamma(p_i)$ and that $\xi$ is affine on each of the intervals $[p_{i},p_{i+1}]$ for each $i\in\Z$. Note that
	\begin{equation}\label{eq:pLsep}
		\enorm{\xi(p_i)-\xi(p_{i-1})}=\enorm{\gamma(p_{i})-\gamma(p_{i-1})}=\omega(L)\qquad\text{ for every $i\in\Z$.}
	\end{equation} 
	Subsequently, applying \eqref{eq:bi_step}, we deduce
	\begin{equation}\label{eq:xispeed}
		\frac{1}{2L}\leq\enorm{\xi'(t)}\leq L \qquad \text{for all $i\in\Z$ and $t\in (p_{i-1},p_{i})$.}
	\end{equation}
	
	Next, we argue that
	\begin{equation}\label{cl:nonadjacent_p_i}
		\enorm{\gamma(p_i)-\gamma(p_j)}>\omega(L) \qquad\text{ whenever $i,j\in\Z$ and $\abs{j-i}\geq 2$.}
	\end{equation}
	To see this, we assume that $i+1<j$. We note that when $i\geq 0$ or $j\leq 0$, the claim is a direct consequence of the definition of $\br{p_i}_{i\in\Z}$. Thus, we are left with $i<0$ and $j>0$. Then $p_i<p_0\leq 0$ and $p_j\geq 0$; therefore, \eqref{eq:sep_plus_minus} applies and proves the claim. 
	
	Now we bound $\infnorm{\gamma-\xi}$ above. Given $t\in\R$, we may find, by \eqref{eq:bi_step}, $i\in\Z$ such that $p_{i-1}<t\leq p_i$. Since $\xi$ is affine on $[p_{i-1}, p_i]$, we see that
	\[
	\enorm{\gamma(t)-\xi(t)}\leq\max\set{\enorm{\gamma(t)-\xi(p_{i-1})}, \enorm{\gamma(t)-\xi(p_i)}},
	\]
	and then $\gamma(p_{j})=\xi(p_{j})$ for every $j\in\Z$, \eqref{eq:bi_step} and $\lip(\gamma)\leq L$ imply that the latter is at most $2L^2\omega(L)$. Hence
	\begin{equation}\label{eq:gm-xi_pL}
		\infnorm{\gamma-\xi}\leq 2L^{2}\omega(L).
	\end{equation}
	
	It remains to provide the lower bilipschitz bound on $\xi$. To this end, we seek to establish the conditions of Lemma~\ref{lemma:pwaff_bil}. In what follows we make frequent use without reference of \eqref{eq:pLsep} and \eqref{cl:nonadjacent_p_i}. Firstly, we note that the angle between consecutive segments of $\xi(\R)$ is at least $\pi/3$; this holds because, for each $i\in\Z$, we have
	
	\[
	\enorm{\xi(p_{i-1})-\xi(p_i)}=\enorm{\xi(p_{i})-\xi(p_{i+1})}\leq\enorm{\xi(p_{i-1})-\xi(p_{i+1})}.
	\]
	We now derive conditions on constants $C_{1},C_{2},C_{3}>0$ so that $\xi$ and $(p_{i})_{i\in\Z}$ satisfy conditions \eqref{speed}--\eqref{gamma_on_nodes} of Lemma~\ref{lemma:pwaff_bil} with $\alpha=\pi/3$. By \eqref{eq:xispeed}, \eqref{speed} is satisfied for any $C_{1}\geq 2L$. To check \eqref{buffer} we let $i,j\in\Z$ with $\abs{j-i}\geq 2$. For $k\in\set{j-1, j}$ we have
	\[
	\omega(L)=\enorm{\xi(p_{i-1})-\xi(p_i)}\leq\min\set{\enorm{\xi(p_{i-1})-\xi(p_k)}, \enorm{\xi(p_{i})-\xi(p_k)}},
	\] 
	and therefore $\dist(\xi(p_{k}),[\xi(p_{i-1}),\xi(p_{i})])\geq \frac{\sqrt{3}\omega(L)}{2}$ for $k=j-1,j$.
	By Lemma~\ref{lemma:twodiscs}, the latter implies that $\dist\br{\sqbr{\xi(p_{j-1}),\xi(p_{j})},\sqbr{\xi(p_{i-1}),\xi(p_{i})}}\geq \frac{\omega(L)}{\sqrt{2}}$. This verifies that condition \eqref{buffer} holds for any $C_{2}\leq \frac{\omega(L)}{\sqrt{2}}$. Finally, to check \eqref{gamma_on_nodes}, note that $\xi|_{\set{p_{i}\colon i\in\Z}}$ is $L$-Lipschitz, since $\xi$ coincides with $\gamma$ on $\set{p_{i}\colon i\in\Z}$. Adding in \eqref{eq:pL2}, \eqref{eq:bi_step} and Claim~\ref{cl:coarse_bilip}, we deduce that $\xi|_{\set{p_{i}\colon i\in\Z}}$ is $2L$-bilipschitz. Finally, with \eqref{eq:pLsep}, we conclude that \eqref{gamma_on_nodes} is satisfied for $C_{1}\geq 2L$ and $C_{3}=\omega(L)$. Hence $\xi$ and $(p_{i})_{i\in\Z}$ satisfy conditions \eqref{speed}--\eqref{gamma_on_nodes} of Lemma~\ref{lemma:pwaff_bil} with $\alpha=\pi/3$, $C_{1}=2L$, $C_{2}=\omega(L)/\sqrt{2}$ and $C_{3}=\omega(L)$, so we get 
	\begin{equation}\label{eq:bil_xi_pL}
		\bilip(\xi)\leq \frac{10(2L)\omega(L)}{\frac{\omega(L)}{\sqrt{2}}}\leq40L.
	\end{equation}
	To finish the proof, we note that the only requirement on $\omega(L)$ is \eqref{eq:pL2}, so we set $\omega(L)=4L^3$.
	Now the assertions of the lemma are given by \eqref{eq:gm-xi_pL} and \eqref{eq:bil_xi_pL}.  
\end{proof}

\section{Separation of discrete sets by a bilipschitz curve.}\label{s:separate}
In this section we prove the following theorem:
\restateseparation
\begin{proof}
	Let $r:=s/100$. For each $x\in X$ let $q_{x}\in r\Z^{2}$ be a choice of point $q\in \br{r\Z\setminus 4r\Z}\times r\Z$ minimising the distance $\enorm{q-x}$ and let 
	\begin{equation}\label{def_Q_x}
		Q_{x}:=\cl{B}_{\infty}(q_{x},4r). 
	\end{equation}
	Note that each $q_{x}$, for $x\in X$, is one of the $4$ vertices of a square of side length $r$ containing $x$ in the standard square tiling of $\R^{2}$ with vertex set $r\Z^{2}$, so, in particular, 
	\begin{equation}\label{eq:qx-x}
		\enorm{q_{x}-x}\leq \sqrt{2}r\leq 2r,\qquad  x\in X.
	\end{equation}
	Moreover, writing $q_{x}=(q_{x,1},q_{x,2})$, we have 
	\begin{equation}\label{eq:qx1_funny}
		q_{x,1}\in r\Z\setminus 4r\Z, \qquad x\in X.
	\end{equation}
	Observe that the $\ell_{\infty}$-balls $Q_{x}:=B_{\infty}(q_{x},4r)$ with $x\in X$ are pairwise disjoint, because $X$ is $s=100r$ separated. In fact, this ensures
	\begin{equation}\label{Qx_separation}
		\dist\br{Q_{x},Q_{y}}\geq 100r-10\sqrt{2}r>80r \qquad \text{whenever $x,y\in X$ and $x\neq y$.}
	\end{equation}
	
	We further note that for each $x\in X$, $\proj_{1}(Q_{x})$ and  $\proj_{2}(Q_{x})$ are both intervals of length $8r$ and from \eqref{eq:qx-x} and $\proj_{2}(X)\subseteq[s,w-s]$  we have
	\begin{equation}\label{Qx_horiz_proj}
		\proj_{2}(Q_{x})\subseteq  \sqbr{94r,w-94r}, \quad x\in X.
	\end{equation}

	We define $\psi\colon [0,16r]\to\R^{2}$ as follows (see Figure~\ref{fig:psi}). Let 
	\begin{equation}\label{eq:psi}
		\psi(t)=\begin{cases}
			(0,0) &t=p_{0}:=0,\\
			(0,w) & t=p_{1}:=\frac{8wr}{w+8r},\\
			(8r,w) & t=p_{2}:=8r,\\
			(8r,0) & t=p_{3}:=8r+\frac{8wr}{w+8r},\\
			(16r,0) & t=p_{4}:=16r.
		\end{cases}
	\end{equation}
	\begin{figure}
		\centering
		\includegraphics{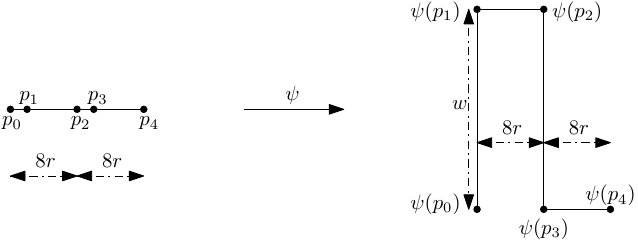}
		\caption{Definition of $\psi$.}
		\label{fig:psi}
	\end{figure}	
	This defines $\psi$ at $5$ points of the interval $[0,16r]$ including at the endpoints of this interval. These $5$ points determine a partition of $[0,16r]$ into $4$ closed subintervals. We now extend $\psi$ to the whole interval $[0,16r]$ by prescribing that $\psi=(\psi_{1},\psi_{2})$ is affine when restricted to each of these 4 subintervals. The points $p_{0},\ldots,p_{4}$ are chosen so that $\enorm{\psi'}$ is constant almost everywhere; we have
	\begin{equation}\label{eq:psi_constantspeed}
		\enorm{\psi'(t)}=\frac{w+8r}{8r} \qquad\text{for all }t\in [0,16r]\setminus\set{p_{0},p_{1},\ldots,p_{4}}.
	\end{equation}
	
	Finally we define $\phi=(\phi_{1},\phi_{2})\colon \R\to\R^{2}$ by 
	\begin{equation}\label{eq:phi}
		\phi(16nr+t)=(16nr+\psi_{1}(t),\psi_{2}(t))\qquad \text{for all }t\in [0,16r],\, n\in\Z.
	\end{equation}
	It will be useful to note 
	\begin{equation}\label{eq:phi1_disp}
		\abs{\phi_{1}(t)-t}\leq 8r \qquad \text{for all }t\in\R.
	\end{equation}
	Let the sequence $(p_{n})_{n\in\Z}$ be defined by 
	\begin{equation}\label{eq:pn}
		p_{4n+l}=16nr+p_{l}; \qquad n\in\Z,\, l\in\set{0,1,2,3},
	\end{equation}
	where $p_{0},\ldots,p_{3}$ are defined in \eqref{eq:psi}. To avoid using double subindices later on, we will on occasion write $p(j)$ instead of $p_{j}$ for $j\in\Z$. We also record the following useful identities, immediate from the construction, for future reference
	\begin{align}
		\phi\br{\sqbr{p_{4n},p_{4n+1}}}&=\set{16nr}\times [0,w],\nonumber\\
		\phi\br{\sqbr{p_{4n+1},p_{4n+2}}}&=[16nr,(16n+8)r]\times\set{w}, \nonumber\\
		\phi\br{\sqbr{p_{4n+2},p_{4n+3}}}&=\set{(16n+8)r}\times [0,w], \nonumber\\
		\phi\br{\sqbr{p_{4n+3},p_{4n+4}}}&=[(16n+8)r,16(n+1)r]\times\set{0},\qquad n\in\Z.\label{phi_images}
	\end{align}
	This also implies that for any $n,m\in\Z$ with $\abs{n-m}\geq 2$ we have
	\begin{equation}\label{eq:phibuffer}
		\dist\br{\phi\br{[p_{m-1},p_{m}]},\phi\br{[p_{n-1},p_{n}]}}\geq 8r.
	\end{equation}
	\begin{claim}\label{claim:phi}
		$\phi$ is $K$-bilipschitz for $K:=\frac{w}{r}$.
	\end{claim}
	\begin{proof}
		We will determine conditions on constants $C_{1},C_{2},C_{3},C_{4}>0$ so that $\phi\colon \R\to\R^{2}$ and $(p_{n})_{n\in\Z}$ satisfy conditions \eqref{speed}--\eqref{far_away} of Lemma~\ref{lemma:pwaff_bil} with $\alpha=\pi/2$. From \eqref{eq:phi} and \eqref{eq:psi_constantspeed} we verify \eqref{speed} for $C_{1}\geq \frac{w+8r}{8r}$. Moreover, \eqref{buffer} for $C_{2}\leq 8r$ is given by \eqref{eq:phibuffer}. To check the last remaining condition \eqref{far_away}, consider $t_{1},t_{2}\in\R$ with $\abs{t_{2}-t_{1}}\geq C_{3}$. Then we use \eqref{eq:phi1_disp} to deduce
		\begin{equation}\label{eq:phi_far_away}
			\enorm{\phi(t_{2})-\phi(t_{1})}\geq \abs{\phi_{1}(t_{2})-\phi_{1}(t_{1})}\geq \abs{t_{2}-t_{1}}-16r\geq \br{1-\frac{16r}{C_{3}}}\abs{t_{2}-t_{1}}.
		\end{equation}
		If $C_{4}\geq \frac{C_{3}}{C_{3}-16r}$, this verifies condition \eqref{far_away}. It only remains to note, recalling $100r=s<1\leq w$, that the choices $C_{1}=\frac{w}{4r}$, $C_{2}=8r$, $C_{3}=32r$ and $C_{4}=2$ satisfy all the conditions we have encountered and so Lemma~\ref{lemma:pwaff_bil} asserts that $\phi$ is $\frac{w}{r}$-bilipschitz.
	\end{proof}
	The constant $K$ defined in Claim~\ref{claim:phi} will appear often in the remainder of the proof.
	
	\begin{claim}\label{claim:axbx}
		For each $x\in X$ there exist $n_{x}\in\Z$ and $i_{x}\in\set{0,2}$ so that the following statements hold:
		\begin{enumerate}[(a)]
			\item\label{ax_bx} $\phi^{-1}\br{Q_{x}}$ is a non-empty closed interval $[a_{x},b_{x}]$ and $[a_{x},b_{x}]\subseteq (p\br{4n_{x}+i_{x}},p\br{4n_{x}+i_{x}+1})$. 
			\item\label{axbx_vertically_opp} The points $\phi(a_{x}),\phi(b_{x})$ are vertically opposite points on the upper and lower horizontal sides of $\partial Q_{x}$.
			
			\item\label{axbx_meet} $r\leq\abs{\proj_{1}(q_{x})-\phi_{1}([p(4n_{x}+i_{x}),p(4n_{x}+i_{x}+1)])}\leq 3r$,
			\item\label{Qx--phi-separation} $\dist\br{[\phi(p_m), \phi(p_{m+1})], Q_x}\geq r$ for every $m\in\Z, m\neq 4n_x+i_x$.
		\end{enumerate}	
	\end{claim}
	\begin{proof}
		Given $x=(x_{1},x_{2})\in X$, choose $k\in\Z$ such that, writing $q_{x}=(q_{x,1},q_{x,2})$, we have $r\leq\abs{q_{x,1}-8kr}\leq 3r$; this is possible due to \eqref{eq:qx1_funny}. Since $x_{2}\in [s,w-s]$ we have $q_{x,2}\in [s-2r,w-s+2r]$. These bounds on the coordinates of $q_{x}$ ensure that $Q_{x}$ intersects the line segment $\set{8kr}\times[0,w]$ and satisfies
		\begin{align*}
			\dist\br{Q_{x},\set{8nr}\times \R}&\geq r \qquad \text{for all }n\in\Z\setminus \set{k},\text{ and}\\
			\dist\br{Q_{x},\R\times\set{0,w}}&\geq r.
		\end{align*} 
		It follows, in view of \eqref{phi_images}, that 
		\begin{equation*}
			\emptyset\neq\phi(\R)\cap Q_{x}=\br{\set{8kr}\times [0,w]}\cap Q_{x}=\phi\br{\sqbr{p_{4n+i},p_{4n+i+1}}}\cap Q_{x},
		\end{equation*}
		where
		\begin{equation*}
			n=n_{x}=\begin{cases}
				\frac{k}{2} & \text{if $k$ is even},\\
				\frac{k-1}{2} & \text{if $k$ is odd},
			\end{cases} \quad\text{ and }\quad i=i_{x}=\begin{cases}
				0 & \text{if $k$ is even},\\
				2 & \text{if $k$ is odd}.
			\end{cases}
		\end{equation*}
		Parts \eqref{axbx_meet} and \eqref{Qx--phi-separation} also follow.
		Note that the endpoints of the vertical line segment 
		\begin{equation*}
			\set{8kr}\times [0,w]=\phi\br{\sqbr{p\br{4n_{x}+i_{x}},p\br{4n_{x}+i_{x}+1}}}
		\end{equation*}
		lie outside of $Q_{x}$, with one above $Q_{x}$ and the other below. Writing $[a_{x},b_{x}]$ for $\phi^{-1}(Q_{x})$, we infer \eqref{ax_bx} and \eqref{axbx_vertically_opp}. 
	\end{proof}

	As a preparation for the definition of $\gamma$, for every $x\in X$ we choose $q_x^+, q_x^-$ to be a pair of diagonally opposite vertices of $Q_x$. By construction of $\phi$, exactly one of the two vertical half-lines emanating from $q_x^+$ or $q_x^-$ does not intersect $\phi(\R)$; without loss of generality, we choose the labels so that the half rays $q_{x}^{+}+[0,\infty)e_{2}$ and  $q_{x}^{-}+(-\infty,0]e_{2}$ do not intersect $\phi(\R)$.
	Additionally, we introduce a curve $\phi_{x}\colon \R\to \R^{2}$ as follows:
	\begin{enumerate}[(A)]
		\item\label{phi_x=phi} $\phi_x(t)=\phi(t)$ for all $t\in \R\setminus (a_{x},b_{x})$.
		\item\label{choosing_side_along_square} $\rest{\phi_x}{[a_{x},b_{x}]}$ is the constant speed parameterisation of one of the two paths along $\partial Q_{x}$ from $\phi_x(a_x):=\phi(a_{x})$ to $\phi_x(b_{x}):=\phi(b_{x})$; we choose the one that contains $q_x^+$ if $x\in Y$, otherwise we choose the one containing $q_x^-$.
	\end{enumerate}
	Obviously, $\phi_x$ is continuous and by Claim~\ref{claim:axbx}\eqref{ax_bx} it is also injective.
	
	If $x\in Y$, then $[x, q_x^-]\cup \br{q_{x}^{-}+(-\infty,0]e_{2}}$ is a path disjoint from $\phi_x(\R)$ by \eqref{choosing_side_along_square}, and thus, shows that $x$ is in the same connected component of $\R^2\setminus\phi_x(\R)$ as $\R\times (-\infty, 0)$. Analogously, we verify that $x\in X\setminus Y$ is in the same connected component as $\R\times (w, \infty)$. Since $\phi_x(\R)\subset\R\times[0, w]$ and $\lim_{t\to\pm\infty}\proj_{1}\phi_x(t)=\pm\infty$, the sets $\R\times (-\infty, 0)$ and $\R\times (w, \infty)$ are in different components of $\R^2\setminus\phi_x(\R)$. Hence, we have showed that for every $x\in X$
	\begin{equation}\label{eq:phix_split}
		x \text{ is in the same component of } \R^2\setminus\phi_x(\R) \text{ as } \R\times (-\infty, 0) \quad\Longleftrightarrow\quad x\in Y.
	\end{equation}
	Using Claim~\ref{claim:axbx}\eqref{axbx_vertically_opp} and \eqref{choosing_side_along_square}, we note that there are precisely two points $c_{x},d_{x}\in (a_{x},b_{x})$ with $c_{x}<d_{x}$ such that $\phi_x(c_{x})$ and $\phi_x(d_{x})$ are the endpoints of one of the vertical sides of $\partial Q_{x}$ and each of the mappings $\phi_x|_{[a_{x},c_{x}]}$, $\phi_x|_{[c_{x},d_{x}]}$ and $\phi_x|_{[d_{x},b_{x}]}$ is affine.	
	
	Moreover, Claim~\ref{claim:axbx}\eqref{axbx_vertically_opp} and \eqref{choosing_side_along_square} allow us to easily compute bounds for the constant speed of $\phi_x|_{[a_{x},b_{x}]}$. The length of the path $\phi_x([a_{x},b_{x}])$ is between $8r$ and $24r$. Thus, we have the central two inequalities below:
	\begin{equation}\label{eq:phi_x_speed}
		\frac{1}{K}\leq\frac{8r}{b_{x}-a_{x}}\leq \enorm{\phi_x'(t)}\leq \frac{24r}{b_{x}-a_{x}}\leq 3K\qquad \text{for all }t\in (a_{x},b_{x})\setminus\set{c_{x},d_{x}}.
	\end{equation}
	The first and last inequalities in the sequence are verified by combining Claim~\ref{claim:axbx}\eqref{axbx_vertically_opp} and Claim~\ref{claim:phi}.
	
	Finally we note that \eqref{phi_x=phi} and \eqref{choosing_side_along_square} imply
	\begin{equation}\label{phi_x-phi}
		\infnorm{\phi_x-\phi}\leq \diam Q_x\leq 16r.
	\end{equation}
	Next, we estimate the bilipschitz constant of $\phi_x$.	
	\begin{claim}\label{claim:phix}
		For every $x\in X$, $\phi_{x}$ is $12K$-bilipschitz.
	\end{claim}
	\begin{proof}
		Recall the parameters $n_{x}\in\Z$ and $i_{x}\in\set{0,2}$ given by Claim~\ref{claim:axbx}\eqref{ax_bx}. We assume that $i_{x}=0$; the case $i_{x}=2$ is similar. Consider the double sided sequence $(p_{n}')_{n\in\Z}$ defined by
		\begin{equation}\label{eq:pn'}
			p_{n}'=\begin{cases}
				p_{n} & \text{if }n\leq 4n_{x},\\
				a_{x} & \text{if }n=4n_{x}+1,\\
				c_{x} & \text{if }n=4n_{x}+2,\\
				d_{x} & \text{if }n=4n_{x}+3,\\
				b_{x} & \text{if }n=4n_{x}+4,\\
				p_{n-4} & \text{if } n\geq 4n_{x}+5
			\end{cases}.
		\end{equation}
		We find conditions on constants $C_{1},C_{2},C_{3},C_{4}>0$ so that $\phi_{x}\colon \R\to\R^{2}$ and $(p_{n}')_{n\in\Z}$ satisfy conditions \eqref{speed}--\eqref{far_away} of Lemma~\ref{lemma:pwaff_bil} with $\alpha=\pi/2$.
		
		By \eqref{eq:phi_x_speed} and Claim~\ref{claim:phi} we have that \eqref{speed} holds for any $C_{1}\geq 3K$. For \eqref{far_away}, let $t_{1},t_{2}\in \R$ with $\abs{t_{2}-t_{1}}\geq C_{3}$. Then, using \eqref{eq:phi_far_away} and \eqref{phi_x-phi} we get
		\begin{multline*}
			\enorm{\phi_{x}(t_{2})-\phi_{x}(t_{1})}\geq\enorm{\phi(t_{2})-\phi(t_{1})}-\enorm{\phi(t_{2})-\phi_{x}(t_{2})}-\enorm{\phi_{x}(t_{1})-\phi(t_{1})}\\
			\geq \br{1-\frac{48r}{C_{3}}}\abs{t_{2}-t_{1}}. 
		\end{multline*}
		Thus, \eqref{far_away} is satisfied if $C_{4}\geq \frac{C_{3}}{C_{3}-48r}$.
		
		It remains to check \eqref{buffer}. For $I_k:=\sqbr{\phi_{x}(p'_{k-1}),\phi_{x}(p'_{k})}$ for $k\in\Z$ and given $m,n\in\Z$ with $\abs{m-n}\geq 2$, we need to derive a lower bound on $\dist\br{I_{m},I_{n}}$.
		\begin{figure}
			\centering
			\includegraphics{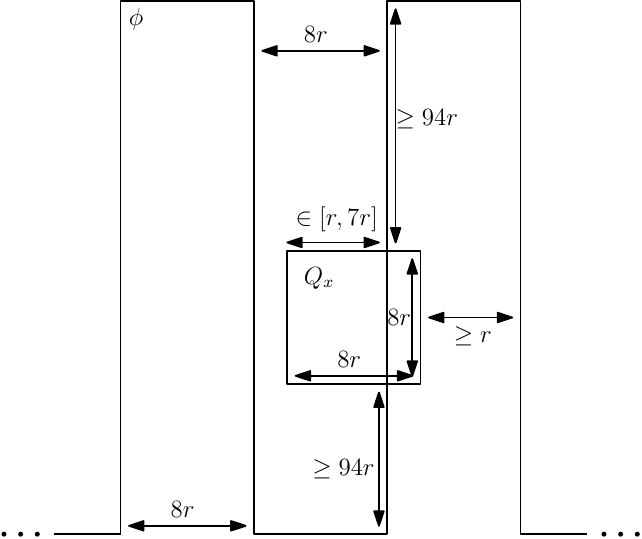}
			\caption{The position of $Q_x$ with respect to $\phi(\R)$.}
			\label{fig:phi_x}
		\end{figure}
		By \eqref{eq:phibuffer}, we may restrict our attention to the case where at least one of the segments $I_{m}$ and $I_{n}$ is a part of $\partial Q_x$. If they are both contained in $\partial Q_x$, then $\dist(I_{m},I_{n})\geq 8r$ by \eqref{def_Q_x}. Otherwise, parts \eqref{axbx_meet} and \eqref{Qx--phi-separation} of Claim~\ref{claim:axbx} ensure that $\dist(I_{m},I_{n})\geq r$.
		Therefore, Lemma~\ref{lemma:pwaff_bil}\eqref{buffer} is satisfied by any $C_{2}\geq r$.
		
		To finish the proof, we note that the choices $C_{1}=3K$, $C_{2}=r$, $C_{3}=96r$ and $C_{4}=2$ satisfy all the conditions we encountered, so Lemma~\ref{lemma:pwaff_bil} with $\alpha=\pi/2$ applied to $\phi_{x}$ and $(p_{n}')_{n\in\Z}$ gives $\bilip(\phi_{x})\leq 12K$.
	\end{proof}

	Finally, we define $\gamma\colon\R \to\R$ by the following two conditions:
	\begin{enumerate}[(I)]
		\item\label{gamma=phi} $\gamma(t)=\phi(t)$ for all $t\in \R\setminus \bigcup_{x\in X}(a_{x},b_{x})$.
		\item\label{going_around_squares} for every $x\in X$ we set $\gamma(t)=\phi_x(t)$ for all $t\in [a_{x},b_{x}]$.
	\end{enumerate}
	Note that \eqref{gamma=phi} and \eqref{phi_x=phi} ensure that $\gamma(t)=\phi_x(t)$ for every $t\in\R\setminus \bigcup_{x\in X}(a_{x},b_{x})$ and every $x\in X$. Thus, several properties of $\phi_x$ are directly inherited by $\gamma.$ First, $\gamma$ is continuous and injective. Second, Claim~\ref{claim:phi} and \eqref{eq:phi_x_speed} yield
	\begin{equation}\label{eq:gamma_speed}
		\frac{1}{K}\leq \enorm{\gamma'(t)}\leq 3K\qquad \text{for almost every }t\in\R
	\end{equation}
	and \eqref{phi_x-phi} translates to
	\begin{equation}\label{gamma-phi}
		\infnorm{\gamma-\phi}\leq 16r.
	\end{equation}
	Lastly, \eqref{gamma=phi} and \eqref{going_around_squares} imply the first inclusion below
	\begin{equation}\label{eq:gamma_image}
		\gamma(\R)\subset\phi(\R)\cup\bigcup_{x\in X}\partial Q_x\subset \R\times[0,w],
	\end{equation}
	while \eqref{phi_images} and \eqref{Qx_horiz_proj} verify the second. Thus, $\gamma(\R)$ separates $\R\times(-\infty, 0)$ from $\R\times(w,\infty)$.

	We are ready to verify the properties \eqref{playground}--\eqref{capture_avoid} of the present theorem. The first part of \eqref{playground} follows from \eqref{gamma-phi} and \eqref{eq:phi1_disp}, whereas the second from \eqref{eq:gamma_image}.
	
	For every $x\in X$, \eqref{eq:qx-x} implies that $\dist(x,\partial Q_x)\geq 2r$. Thus, property~\eqref{stayaway} is a consequence of the first inclusion in \eqref{eq:gamma_image}, Claim~\ref{claim:axbx}\eqref{ax_bx}, \eqref{going_around_squares} and \eqref{choosing_side_along_square}.
	
	To see that $\gamma$ satisfies \eqref{capture_avoid}, we fix $x\in Y$ (the argument for $X\setminus Y$ is symmetric). Let $V_x^-$ denote the connected component of $\R^2\setminus\phi_x(\R)$ containing $\R\times(-\infty, 0)$ and recall, from \eqref{eq:phix_split}, that $x\in V_{x}^{-}$. Thanks to parts \eqref{ax_bx} and \eqref{Qx--phi-separation} of Claim~\ref{claim:axbx} and the separation between $Q_y, Q_{y'}$ asserted by \eqref{Qx_separation} for all $y, y'\in X, y\neq y'$, every path in $V_x^-$ intersecting some $Q_y, y\in X\setminus{x}$, can be slid along the boundary of $Q_y$ so that the perturbed path avoids $Q_y$ while staying in $V_x^-$ and so that we do not introduce any new intersection with any of $Q_{y'}, y'\neq y$, in the process. Therefore, the set $U:=V_x^-\setminus\bigcup_{y\in X\setminus\set{x}} Q_y$ is connected. 
	Note that $x\in U$. As $\gamma(\R)\subseteq\phi_x(\R)\cup\bigcup_{y\in X\setminus\set{x}} Q_y$, it follows that either $U\subset W^-$, or $U\subset W^+$. Since $\R\times(-\infty, 0)\subset U$, we infer that $U\subset W^-$. This proves \eqref{capture_avoid}.
	
	It remains to show that $\gamma$ is bilipschitz and estimate its bilipschitz constant:
	\begin{claim}\label{claim:gammaBL}
		The mapping $\gamma\colon\R\to\R^{2}$, defined by \eqref{gamma=phi} and \eqref{going_around_squares}, is $12K$-bilipschitz.
	\end{claim}
	\begin{proof}
		Since $\gamma$ is piecewise affine, \eqref{eq:gamma_speed} implies that $\gamma$ is $3K$-Lipschitz. Given $t_{1},t_{2}\in\R$ it remains to determine conditions on a constant $E\geq 1$ so that
		\begin{equation}\label{eq:lower_BL_gamma}
			\enorm{\gamma(t_{2})-\gamma(t_{1})}\geq \frac{1}{E}\abs{t_{2}-t_{1}}.
		\end{equation}
		If $\set{t_{1},t_{2}}$ intersects at most one interval of the form $[a_{x},b_{x}]$ with $x\in X$, then \eqref{eq:lower_BL_gamma} for $E\geq 12K$ follows from Claim~\ref{claim:phix}. Otherwise there exist $x,y\in X$ with $x\neq y$ such that $t_{1}\in [a_{x},b_{x}]$ and $t_{2}\in [a_{y},b_{y}]$. Let $M=M(w,s)>0$ be a parameter, depending only on $w$ and $s$, to be determined later. If $\abs{t_{2}-t_{1}}\leq M$ then, by \eqref{Qx_separation},  \eqref{going_around_squares} and \eqref{choosing_side_along_square} we have
		\begin{align*}
			\enorm{\gamma(t_{2})-\gamma(t_{1})}\geq\dist(Q_x, Q_y)&\geq \frac{80r}{M}\abs{t_{2}-t_{1}},
		\end{align*}
		and if $\abs{t_{2}-t_{1}}>M$ then by \eqref{gamma-phi}
		\begin{equation*}
			\enorm{\gamma(t_{2})-\gamma(t_{1})}\geq \enorm{\phi(t_{2})-\phi(t_{1})}-32r\geq \br{\frac{1}{K}-\frac{32r}{M}}\abs{t_{2}-t_{1}}.
		\end{equation*}
		The last two lower bounds are equal with value $\frac{80}{112K}$ when we set $M=112Kr$. Thus, in the case $t_{1}\in [a_{x},b_{x}]$ and $t_{2}\in [a_{y},b_{y}]$ with $x\neq y$, we have that \eqref{eq:lower_BL_gamma} holds for any $E\geq \frac{112K}{80}$. 
		
		The strictest condition on $E$ which we needed for \eqref{eq:lower_BL_gamma} over all cases was $E\geq 12K$.
	\end{proof}
	To verify the bound on $\bilip(\gamma)$ in the statement of Theorem~\ref{thm:separation} we plug in the value of $K$ from Claim~\ref{claim:phi} into the expression of Claim~\ref{claim:gammaBL}:
	\begin{equation*}
		12K=\frac{12w}{r}= \frac{1200w}{s}. \qedhere
	\end{equation*}
\end{proof}

\section{Extension to horizontal lines.}\label{s:ext_to_horizont_lines}
The purpose of the present section is to prove the following statement:
\restatestrips
We immediately reduce Theorem~\ref{thm:strips} to the following simpler version:
\begin{lemma}\label{lemma:one_horizontal_line}
	Let $L\geq 1$, $k\in \Z$ and $f\colon \Z^{2}\to\R^{2}$ be an $L$-bilipschitz mapping. Then there exists a $10^{39}L^{15}$-bilipschitz extension
	\begin{equation*}
		F\colon\Z^{2}\cup \br{\R\times \set{k+\frac{1}{2}}}\to \R^{2}
	\end{equation*}
	of $f$ such that $\bilip\br{F|_{\R\times\set{k+\frac{1}{2}}}}\leq 10^{24}L^{11}$ and for any $x,y\in\Z^{2}$ we have
	\begin{equation*}
		\begin{array}{c}
			\text{$x$ and $y$ belong to the same connected component of $\R^{2}\setminus \br{\R\times \set{k+\frac{1}{2}}}$}\\
			\Updownarrow \\
			\text{$f(x)$ and $f(y)$ belong to the same connected component of $\R^{2}\setminus F\br{\R\times \set{k+\frac{1}{2}}}$}.
		\end{array}
	\end{equation*}
\end{lemma}
\begin{lemma}\label{lemma:hor_line_red}
	Lemma~\ref{lemma:one_horizontal_line} implies Theorem~\ref{thm:strips}. 
\end{lemma}
\begin{proof}
	We prove this statement in general dimension $d\geq 2$ in \cite[Lemmas~6.2\,\&\,6.3]{DK_ddim}; the proof goes along the `expected path'. What follows is a concise proof, referring to \cite{DK_ddim} for some of the details. In particular, we explain how the reader can quickly verify all of the constants associated to Lemma~\ref{lemma:hor_line_red}, using \cite{DK_ddim}. 
	
	Let $L$, $T$ and $f$ be given by the hypotheses of Theorem~\ref{thm:strips} and let $F_{i}$, for each $i\in\Z$, be given by the conclusion of Lemma~\ref{lemma:one_horizontal_line} with $k=Ti$. Then $F:=\bigcup_{i\in\Z}F_{i}\colon \Z^{2}\cup\br{\R\times \br{T\Z+\frac{1}{2}}}\to \R^{2}$ is well-defined and $T$ is large enough to ensure $F$ is $10^{39}L^{15}$-bilipschitz with $\bilip\br{F|_{\R\times\br{T\Z+\frac{1}{2}}}}\leq 10^{24}L^{11}$; to verify these constants apply \cite[Lem\-ma~6.2]{DK_ddim}, with $d=2$, $J=10^{39}L^{15}$ and $K=10^{24}L^{11}$. Finally \eqref{Vi} and \eqref{images_in_Vi} follow from the definition of $F$ and the last stated property of $F_{i}$ from the conclusion of Lemma~\ref{lemma:one_horizontal_line}; this is confirmed by \cite[Lemma~6.3]{DK_ddim}. 
\end{proof}	
	\begin{lemma}\label{lemma:brute_force_bilipschitz}
		Let $1\leq L\leq K$, $M\geq 10K$, $\theta\in (0,1/10)$, $f\colon \Z^{2}\to \R^{2}$ be an $L$-bilipschitz mapping and $\gamma\colon \R\to\R^{2}$ be an $M$-bilipschitz mapping. Suppose that $\enorm{\gamma(i)-f(i,0)}\leq K$ for all $i\in\Z$ and $\enorm{\gamma(t)-f(x)}\geq \theta$ for all $t\in\R$ and $x\in\Z^{2}$. Then the mapping $F\colon \Z^{2}\cup\br{\R\times\set{1/2}}\to\R^{2}$ defined by
		\begin{equation*}
			F(x)=f(x)\quad \text{for all }x\in \Z^{2}, \quad\text{and}\quad F(t,1/2)=\gamma(t)\quad \text{whenever }t\in\R,
		\end{equation*}
		is $\frac{ML}{\theta}$-bilipschitz.
	\end{lemma}
	\begin{proof}
		It is clear that the $\frac{ML}{\theta}-$bilipschitz condition is satisfied between pairs of points $x$, $y$ which either both belong to $\Z^{2}$ or both belong to $\R\times\set{\frac{1}{2}}$. So assume $x\in\Z^{2}$, $t\in\R$ and $y=\br{t,\frac{1}{2}}\in\R\times\set{\frac{1}{2}}$. Then, for $t'\in\Z$ chosen so that $\abs{t'-t}\leq 1/{2}$, we have
		\begin{align*}
			\enorm{F(y)-F(x)}&\leq \enorm{\gamma(t)-\gamma(t')}+\enorm{\gamma(t')-f(t',0)}+\enorm{f(t',0)-f(x)}\\
			&\leq \frac{M}{2}+K+L\br{\enorm{y-x}+1}\\
			&\leq \br{M+2K+3L}\enorm{y-x}\leq \frac{ML}{\theta}\enorm{y-x}.
		\end{align*}
		Now we seek the lower bound on $\enorm{F(y)-F(x)}$. Let $P=P(L,K,M,\theta)>1$ be a quantity to be determined later in the proof. If $\enorm{y-x}\geq P$, observe
		\begin{align*}
			\enorm{F(y)-F(x)}&\geq \enorm{f(t',0)-f(x)}-\enorm{f(t',0)-\gamma(t')}-\enorm{\gamma(t')-\gamma(t)}\\
			&\geq \frac{1}{L}\br{\enorm{y-x}-1}-K-\frac{M}{2}\\
			&\geq \br{\frac{1}{L}-\frac{\frac{1}{L}+K+\frac{M}{2}}{P}}\enorm{y-x}\geq \br{\frac{1}{L}-\frac{3M}{4P}}\enorm{y-x}.
		\end{align*}	
		On the other hand, if $\enorm{y-x}\leq P$ then 
		\begin{equation*}
			\frac{\enorm{F(y)-F(x)}}{\enorm{y-x}}=\frac{\enorm{\gamma(t)-f(x)}}{\enorm{y-x}}\geq \frac{\theta}{P}. 
		\end{equation*}
		Thus, we establish the optimal Lipschitz estimate for $F^{-1}$ when
		\begin{equation*}
			\frac{1}{L}-\frac{3M}{4P}=\frac{\theta}{P},\qquad\text{or, equivalently,}\qquad P=L\br{\theta+\frac{3M}{4}}.
		\end{equation*}
		Plugging this value of $P$ into the estimates above, we verify that $F$ is $\frac{ML}{\theta}$-bilipschitz.
	\end{proof}
	
	\begin{proof}[Proof of Theorem~\ref{thm:strips} and Lemma~\ref{lemma:one_horizontal_line}]
		We give a proof of Lemma~\ref{lemma:one_horizontal_line}; together with Lemma~\ref{lemma:hor_line_red} this also provides a proof of Theorem~\ref{thm:strips}.
		We may assume that $k=0$. Let Kirszbraun's Theorem~\cite{Kirszbraun} provide an $L$-Lipschitz mapping $\mu\colon \R\to\R^{2}$ such that $\mu(i)=f(i,0)$ for all $i\in\Z$. Next, let $\xi\colon \R\to\R^{2}$ be the $40L$-bilipschitz mapping satisfying $\enorm{\xi(t)-\mu(t)}\leq 8L^{5}$ for all $t\in\R$, given by Proposition~\ref{prop:initial_red_shore}. We record that 
		\begin{equation*}
			\enorm{\xi(i)-f(i,0)}\leq 8L^{5}\qquad \text{ for all $i\in\Z$.}
		\end{equation*}
		Suppose that $x\in\Z^{2}$, $P>0$, $\dist\br{x,\R\times\set{0}}\geq P$ and that $x'\in\Z^{2}$ with $\enorm{x'-x}\leq 1$. We determine a large enough value for $P$ which ensures that $f(x)$ and $f(x')$ belong to the same connected component of $\R^{2}\setminus \xi(\R)$. Note $\enorm{f(x')-f(x)}\leq L$, whilst for any $t\in\R$ we may write
		\begin{align*}
			\enorm{f(x)-\xi(t)}&\geq \enorm{f(x)-f(\floor{t},0)}-\enorm{\xi(t)-\xi(\floor{t})}-\enorm{\xi(\floor{t})-f(\floor{t},0)}\\
			&\geq \frac{P}{L}-40L-8L^{5}.
		\end{align*} 
		If $P:=10^{2}L^{6}$ then the latter quantity is at least $50L^{5}$. Since this inequality holds for an arbitrary $t\in\R$, we get that $\dist(f(x),\xi(\R))\geq 50L^{5}>\enorm{f(x')-f(x)}$, implying that $f(x)$ and $f(x')$ belong to the same connected component of $\R^{2}\setminus \xi(\R)$.
		
		We may now conclude that the set $f\br{\Z\times\Z_{\geq P}}$ belongs to one connected component of $\R^{2}\setminus \xi(\R)$, whilst the set $f\br{\Z\times\Z_{\leq -P}}$ belongs to the other. Apply Theorem~\ref{lemma:normalise_step} to $\xi$ to get a bilipschitz mapping $\Phi\colon \R^{2}\to\R^{2}$ such that 
		\begin{align}
			\Phi \circ \xi(t)=(t,0)\quad\text{ for all $t\in\R$,}&\qquad \bilip(\Phi)\leq 2000(40L)\leq 10^{5}L,\nonumber\\ 
			\Phi\circ f\br{\Z\times \Z_{\geq P}}\subseteq \R\times \br{0,\infty}\qquad&\text{and}\qquad\Phi\circ f\br{\Z\times\Z_{\leq -P}}\subseteq \R\times \br{-\infty,0}. \label{eq:Phi_f_split}
		\end{align}
		We now determine a large enough value of $R>0$ so that 
		\begin{equation}\label{eq:Phi_f_uncert}
			\Phi\circ f\br{\Z\times \br{\Z\cap [-P,P]}}\subseteq \R\times \sqbr{-R,R}.
		\end{equation}
		Let $x\in \Z\times \br{\Z\cap [-P,P]}$. Then 
		\begin{multline*}
			\dist(\Phi\circ f(x),\R\times \set{0})\leq \lip(\Phi)\dist(f(x),\xi(\R))\\
			\leq \lip(\Phi)\br{\enorm{f(x)-f(\proj_{1}(x),0)}+\enorm{f(\proj_{1}(x),0)-\xi(\proj_{1}(x))}}\\
			\leq \lip(\Phi)\br{LP+8L^{5}}\leq 2\cdot 10^{7}L^{8}.
		\end{multline*}
		Thus, setting $R:=2\cdot 10^{7}L^{8}$, we verify \eqref{eq:Phi_f_uncert}.
		
		Our aim is now to apply Theorem~\ref{thm:separation} on the strip $\R\times [-R,R]$ to obtain a curve separating $\Phi\circ f(\Z\times \Z_{\geq 1})$ and $\Phi\circ f(\Z\times \Z_{\leq 0})$, however, note that the points of the separated set in Theorem~\ref{thm:separation} must have at least some uniform positive distance $s>0$ from the boundary of the strip. Thus, we need to apply a further bilipschitz transformation $\Gamma\colon\R^{2}\to \R^{2}$ so that $\Gamma\circ \Phi\circ f(\Z^{2})$ does not intersect two thin strips centred at the lines $\R\times\set{R}$ and $\R\times\set{-R}$. At the same time, we want the transformation by $\Gamma$ to preserve the properties in \eqref{eq:Phi_f_split} and \eqref{eq:Phi_f_uncert}.

		To construct $\Gamma$ we will use Theorem~\ref{thm:simultaneous_switching} to move points of $\Phi\circ f(\Z^{2})$ lying too close to the two lines $\R\times\set{-R,R}$. For $\eta:=\frac{1}{10^{6}L^{2}}$ and each $x\in \Phi\circ f(\Z^{2})$ let 
		\begin{equation*}
			y_{x}=\begin{cases}
				x+\frac{\eta}{5}e_{2} & \text{if }\proj_{2}(x)\in \sqbr{-R,-R+\frac{\eta}{5}}\cup \left(R,R+\frac{\eta}{5}\right]\\
				x-\frac{\eta}{5}e_{2}& \text{if }\proj_{2}(x)\in \left[-R-\frac{\eta}{5},-R\right)\cup \sqbr{R-\frac{\eta}{5},R},\\
				x & \text{ otherwise,}
			\end{cases}
		\end{equation*}
		and let $r_{x}:=\frac{\eta}{10}$ and $\mf{U}_{x}:=B([x,y_{x}],r_{x})$. Since $\Phi\circ f(\Z^{2})$ is $\eta$-separated, and $\mf{U}_{x}\subseteq B(x,\enorm{y_{x}-x}+r_{x})\subseteq B(x,\eta/2)$ for each $x\in \Phi\circ f(\Z^{2})$, we have that the family $(\mf{U}_{x})_{x\in \Phi\circ f(\Z^{2})}$ is pairwise disjoint. Let the $16$-bilipschitz mapping $\Gamma\colon \R^{2}\to\R^{2}$, with 
		\begin{equation*}
			\Gamma(x)=y_{x}\text{ for all }x\in\Phi\circ f(\Z^{2})\quad\text{ and }\quad \Gamma(p)=p\text{ for all }p\in \R^{2}\setminus \bigcup_{x\in \Phi\circ f(\Z^{2}),\, y_{x}\neq x}\mf{U}_{x},
		\end{equation*}
		be given by Theorem~\ref{thm:simultaneous_switching} and set $\Upsilon:=\Gamma\circ \Phi$ so that
		\begin{equation*}
			\Upsilon\circ \xi(t)=(t,0)\quad\text{for all }t\in\R,\quad\text{ and}\qquad\bilip(\Upsilon)\leq 4^{2}\cdot 10^{5}L.
		\end{equation*}
		Since $\Gamma(x)=y_{x}$ and $x$ belong to the same connected component of $\R^{2}\setminus \br{\R\times\set{0}}$ for every $x\in \Phi\circ f(\Z^{2})$, \eqref{eq:Phi_f_split} becomes
		\begin{equation}\label{eq:Ups_f_split}
			\Upsilon\circ f\br{\Z\times \Z_{\geq P}}\subseteq \R\times \br{0,\infty}\qquad\text{and}\qquad\Upsilon\circ f\br{\Z\times\Z_{\leq -P}}\subseteq \R\times \br{-\infty,0}.
		\end{equation}
		Moreover, $\Gamma(x)\in\R\times [-R,R]$ for every $x\in \Phi\circ f(\Z^{2})\cap \br{\R\times\sqbr{-R,R}}$, so \eqref{eq:Phi_f_uncert} becomes
		\begin{equation}\label{eq:Ups_f_uncert}
			\Upsilon\circ f(\Z\times \br{\Z\cap [-P,P]})\subseteq \R\times [-R,R].
		\end{equation}
		Using $\norm{\Gamma(x)-x}=\enorm{y_{x}-x}\leq \eta/5$ for every $x\in \Phi\circ f(\Z^{2})$, we get that $\Upsilon \circ f(\Z^{2})$ is $\frac{3\eta}{5}$-separated. Note also that 
		\begin{equation}\label{eq:clear_out}
			\proj_{2}\circ \Upsilon\circ f(\Z^{2}) \cap\br{\br{-R-\frac{\eta}{5},-R+\frac{\eta}{5}}\cup\br{R-\frac{\eta}{5},R+\frac{\eta}{5}}}=\emptyset.
		\end{equation}
		Thus, the conditions of Theorem~\ref{thm:separation} (appropriately shifted) are satisfied for 
		\begin{align*}
			s:=\frac{\eta}{5}=\frac{2}{10^{7}L^{2}},&\qquad w:=2R,\\
			X:=\Upsilon\circ f(\Z^{2})\cap \br{\R\times [-R,R]},&\qquad Y:=\Upsilon\circ f(\Z\times \Z_{\leq 0})\cap \br{\R\times [-R,R]}.
		\end{align*}
		We obtain a bilipschitz mapping $\gamma=(\gamma_{1},\gamma_{2})\colon \R\to\R^{2}$ such that  
		\begin{align*}
			&\bilip(\gamma)\leq \frac{1200\cdot 2R}{\frac{2}{10^{7}L^{2}}}= 2400\cdot 10^{14}L^{10}\leq \frac{1}{4}10^{18}L^{10},\\
			&\abs{\gamma_{1}(t)-t}\leq \frac{s}{4}=\frac{1}{2\cdot 10^{7}L^{2}}\quad \text{and} \quad \abs{\gamma_{2}(t)}\leq R=2\cdot 10^{7}L^{8}\quad \text{for all $t\in\R$,}\qquad \text{and}\\
			&\dist(\Upsilon\circ f(x),\gamma(\R))\geq \frac{2s}{100}= \frac{4}{10^{9}L^{2}}  \qquad \text{for all }x\in\Z^{2}.
		\end{align*}
		The last stated property for $x\in\Z^{2}$ comes from the conclusion of Theorem~\ref{thm:separation}, if $\Upsilon\circ f(x)\in \R\times\sqbr{-R,R}$, and from \eqref{eq:clear_out}, otherwise. Moreover, the mapping $\gamma$ we get from Theorem~\ref{thm:separation} has one further property: letting $W_{+}$ and $W_{-}$ denote the connected components of $\R^{2}\setminus \gamma(\R)$ containing $\R\times (R,\infty)$ and $\R\times (-\infty,-R)$ respectively, we have
		\begin{equation}\label{eq:gamma_split}
			\Upsilon\circ f(\Z\times\Z_{\leq 0})\cap\br{\R\times [-R,R]}\subseteq W_{-}\quad \text{and}\quad \Upsilon\circ f(\Z\times \Z_{\geq 1})\cap \br{\R\times [-R,R]}\subseteq W_{+}.
		\end{equation}
		Define $F\colon \Z^{2}\cup \br{\R\times\set{\frac{1}{2}}}\to \R^{2}$ by
		\begin{equation*}
			F(x)=f(x)\quad \text{for all }x\in \Z^{2}, \quad\text{and}\quad F(t,1/2)=\Upsilon^{-1}\circ \gamma(t)\quad \text{whenever }t\in\R.
		\end{equation*}
		Combining \eqref{eq:gamma_split}, \eqref{eq:Ups_f_split} and \eqref{eq:Ups_f_uncert}, we deduce that $f(\Z\times\Z_{\leq 0})\subseteq \Upsilon^{-1}(W_{-})$ and $f(\Z\times\Z_{\geq 1})\subseteq \Upsilon^{-1}(W_{+})$. 
		Finally, the inequalities
		\begin{align*}
			\bilip(\Upsilon^{-1}\circ \gamma)&\leq 4^{2}\cdot 10^{5}L\cdot \frac{1}{4}10^{18}L^{10}\leq10^{24}L^{11}=:M,\\
			\dist(f(x),\Upsilon^{-1}\circ \gamma(\R))&\geq \frac{1}{\bilip(\Upsilon)}\dist\br{\Upsilon\circ f(x),\gamma(\R)}\\
			&\geq \frac{1}{4^{2}\cdot 10^{5}L}\cdot\frac{4}{10^{9}L^{2}}\geq\frac{1}{10^{15}L^{3}}=:\theta \quad\text{for all $x\in\Z^{2}$,}\quad\text{and}\\
			\enorm{\Upsilon^{-1}\circ \gamma(i)-f(i,0)}&\leq \bilip(\Upsilon)\enorm{\gamma(i)-\Upsilon\circ \xi(i)}+\enorm{\xi(i)-f(i,0)}\\
			&\leq \bilip(\Upsilon)\br{\abs{\gamma_{1}(i)-i}+\abs{\gamma_{2}(i)}}+\enorm{\xi(i)-f(i,0)}\\
			&\leq 4^{2}\cdot 10^{5}L\cdot\br{1+2\cdot 10^{7}L^{8}}+8L^{5}
			\leq  10^{14}L^{9}=:K\quad \text{for all $i\in\Z$,}
		\end{align*}
		allow us to apply Lemma~\ref{lemma:brute_force_bilipschitz} to establish $\bilip(F)\leq \frac{ML}{\theta}=10^{39}L^{15}$. We also get $\bilip\br{F|_{\R\times\set{\frac{1}{2}}}}=\bilip(\Upsilon^{-1}\circ \gamma)\leq 10^{24}L^{11}$.
	\end{proof}

\section{Bilipschitz mappings on a strip.}\label{ext_to_strip}
The aim of the present section is to build upon Theorem~\ref{thm:Kovalev} to show that bilipschitz mappings can be extended from a boundary of a strip to the whole strip in the plane:
\restateexttostrip

As a preparation for the proof of Theorem~\ref{thm:ext_to_strip}, we note that while Theorem~\ref{thm:Kovalev} is stated for mappings on the unit square, we can readily extend it to any domain of the form $T([0,1]^2)$, where $T$ is a bilipschitz mapping $[0, 1]^2\to\R^2$, by composing the relevant mappings with $T$ or $T^{-1}$. For convenience, we state it formally.
\begin{prop}\label{p:generalised_DP}
	Let $K,L\geq 1$ and $T\colon [0,1]^{2}\to \R^{2}$ and $f\colon \partial\br{T\br{[0,1]^{2}}}\to\R^{2}$ be $K$-bilipschitz and $L$-bilipschitz mappings respectively. Then there is a $2\cdot 10^{16}K^{2}L$-bilipschitz mapping $F\colon T\br{[0,1]^{2}}\to \R^{2}$ extending $f$.
\end{prop}
\begin{proof}
	We define $g(x):=f(T(x))$, which is a $KL$-bilipschitz mapping $\partial[0,1]^2\to\R^2$. By Theorem~\ref{thm:Kovalev} we get its $2\cdot 10^{16}KL$-bilipschitz extension $G$ to $[0,1]^2$. Setting $F(x):=G(T^{-1}(x))$ we obtain the desired extension of $f$.
\end{proof}

We also recall the gluing operation from Definition~\ref{def:gluing}. Lemma~\ref{lem:gluing_lip} about gluing of Lipschitz mappings will be used frequently, and thus, in the present section the letter $z$ always refers to $z(\cdot, \cdot)$ from its statement.

While it is clear that every trapezium is a bilipschitz image of $[0, 1]^2$, the purpose of the following lemma is to record an explicit bound on the bilipschitz constant of such a bijection.
\begin{lemma}\label{lem:trapezium_bilip}
	Let $P$ be a trapezium of height $h$, with bases of length $b_1$ and $b_2$, with $b_1\leq b_2$, 
	and legs of lengths $l_1, l_2$, where $l_1\leq l_2$.
	Then there is a bilipschitz bijection $T\colon[0,1]^2\to P$ with
	\[
	\lip(T)\leq b_2+l_2 \quad\text{and}\quad\lip\br{T^{-1}}\leq\frac{\sqrt{b_2^2+\frac{3 l_2^2}{2}}}{b_1 h}.
	\]
\end{lemma}
\begin{proof}
	Up to a rotation, we may assume that the bases of $P$ are horizontal.
	The mapping $T$ is then defined to map the vertices of $[0, 1]^2$ onto the vertices of $P$ in the natural order (i.e., the left bottom vertex of $[0, 1]^2$ is mapped onto the left bottom vertex of $P$, and so on). We split $[0, 1]^2$ into two triangles $\Delta_1$ and $\Delta_2$ using one of the diagonals of $[0,1]^2$. Now we define $T$ as the unique affine extension of $T$ to each of these two triangles. We denote the length of the common side of $T(\Delta_1)$ and $T(\Delta_2)$ by $d$. Note that this side is a diagonal of $P$. The indices are chosen in such a way that $T(\Delta_1)$ has sides of lengths $b_1, l_1, d$.

	It is enough to bound the bilipschitz constants of $\rest{T}{\Delta_1}$ and $\rest{T}{\Delta_2}$ separately, since these mappings and their inverses satisfy the assumptions of Lemma~\ref{lem:gluing_lip} with $C=1$; to see this, we just note that for every $x_i\in\Delta_i$ and $y_i\in T(\Delta_i), i\in\set{1, 2}$, it suffices to take $z(x_1, x_2)$ as the intersection of $[x_1, x_2]$ with the common side of $\Delta_1$ and $\Delta_2$, and similarly, $z(y_1, y_2)$ as the intersection of $[y_1, y_2]$ with the common side of $T(\Delta_1)$ and $T(\Delta_2)$. We estimate the bilipschitz constant of $T$ on $\Delta_2$; the other case is the same. Up to a reflection, we may assume that the chosen diagonal goes from the top left vertex of $[0, 1]^2$ to the bottom right vertex and that $\Delta_2$ is beneath it. Up to a translation, we may also assume that $T(0, 0) = (0, 0)$ and $T(1, 0)=(b_2, 0)$. Moreover, we write $T(0, 1)$ as $(a_1b_2, a_2b_2)$, with appropriately chosen $a_1, a_2\in\R$. Hence, we can view $T$ as a composition of a scaling by $b_2$ and a linear mapping $\widetilde{T}=\begin{psmallmatrix} 1 & a_1 \\ 0 & a_2 \end{psmallmatrix}$. The defining relations for $a_1, a_2$ are  $a_2=\frac{h}{b_2}>0$ and $a_1^2+a_2^2=\frac{l_2^2}{b_2^2}$, as $h$ is the height of $T(\Delta_2)$ and $l_2$ the length of $T(0, 1)$. It remains to bound $\opnorm{\widetilde{T}}$ and $\opnorm{\widetilde{T}^{-1}}$.
	
	We fix $x, y$ such that $x^2+y^2=1$ and bound $\enorm{\widetilde{T}(x,y)}$ from above:
	\[
	\enorm{\widetilde{T}(x,y)}^2=(x+a_1y)^2+a_2^2y^2=x^2+2a_1xy+(a_1^2+a_2^2)y^2\leq 1+\abs{a_1}+\frac{l_2^2}{b_2^2}\leq\frac{b_2^2+b_2l_2+l_2^2}{b_2^2},
	\]
	where we used that $2xy\leq x^2+y^2=1$, that $x^2, y^2\leq 1$ and also that $a_{1}^{2}\leq \frac{l_{2}^{2}}{b_{2}^{2}}$. We conclude that   $\lip\br{\rest{T}{\Delta_2}}= b_2\opnorm{\widetilde{T}}\leq \sqrt{b_2^2+b_2l_2+l_2^2}<b_2+l_2$.
	
	By a direct calculation, we see that $\widetilde{T}^{-1}=\begin{psmallmatrix} 1 & -\frac{a_1}{a_2} \\ 0 & \frac{1}{a_2} \end{psmallmatrix}$. The same way as above, we get
	\begin{align*}
		\enorm{\widetilde{T}^{-1}(x,y)}^2=\frac{1}{a_2^2}\br{(a_2x-a_1y)^2+y^2}&\leq\frac{a_2^2+\abs{a_1a_2}+a_1^2+1}{a_2^2}\\
		&\leq\frac{1+\frac{3}{2}(a_2^2+a_1^2)}{a_2^2}\leq\frac{b_2^2+\frac{3l_2^2}{2}}{h^2}.
	\end{align*}
	Therefore, we obtain that $\lip\br{\rest{T^{-1}}{T(\Delta_2)}}\leq\sqrt{\frac{b_2^2+\frac{3l_2^2}{2}}{b_2^2h^2}}$.
	
	By an analogous computation for $\Delta_1$, we get that $\lip(T|_{\Delta_{1}})\leq b_{1}+l_{1}$ and  $\lip\br{\rest{T^{-1}}{T(\Delta_1)}}\leq\sqrt{\frac{b_1^2+\frac{3l_1^2}{2}}{b_1^2h^2}}$. Clearly, the bounds for $\lip(T)$ and $\lip(T^{-1})$ in the statement of the lemma are upper bounds  for $\lip\br{\rest{T}{\Delta_i}}$ and $\lip\br{\rest{T^{-1}}{T(\Delta_i)}}$, $i\in\set{1, 2}$, respectively. 
	\qedhere
\end{proof}
Let $a,w,{p}\in\R^2$ such that $a\neq {p}$ and $w\neq {p}$. We denote by $\gamma(w,{p},a)$ the size of the angle at ${p}$ in the (possibly degenerate) triangle $a{p}w$. Obviously, $\gamma$ is a continuous function. We will need a lower bound on $\gamma(w, {p}, a)$ on triples $a, w, {p}$ such that $w, {p}$ lie in a certain compact convex set $\mc{K}$ and $a\notin\inter\mc{K}$.
\begin{lemma}\label{fact:bisector}
	Let $\mc{K}$ be a compact convex set in $\R^2$ and $w,{p}\in\mc{K}, w\neq {p}$ such that the open line segment $(w,{p})$ lies in $\inter\mc{K}$. Then $\gamma(w, {p}, \cdot)$ constrained to the set \[
	A_{w{p}}:=\set{a\in\R^2\setminus\br{\inter\mc{K}\cup\set{{p}}}\colon \enorm{a-{p}}\leq\enorm{a-w}}\]
	attains its minimum at an intersection of $\partial\mc{K}$ with the bisector of $[w,{p}]$.
\end{lemma}
\begin{proof}
	By the convexity of $\mc{K}$ and the assumptions on $w,{p}$, the bisector of $[w,{p}]$ intersects $\partial\mc{K}$ at exactly two points $a_1, a_2$. As $a_1, a_2\neq {p}$, we see that $a_1, a_2\in A_{w{p}}$. Now we take $a\in A_{w{p}}$ and assume for contradiction that
	\begin{equation}\label{eq:lem_bisector}
		\gamma(w, {p}, a)<\min\set{\gamma(w, {p}, a_1), \gamma(w, {p}, a_2)}.
	\end{equation}
	This means that $a$ lies in the interior of the convex cone containing $w$ and bounded by the lines $a_1{p}$, $a_2{p}$; this cone together with the bisector of $[w,{p}]$ defines a triangle $T$. By our assumptions, $a\in T$. On the other hand, $T\subset\mc{K}$ by the convexity of $\mc{K}$. As $a\in T\cap A_{w{p}}\subseteq\partial T$ and since $(a_1, a_2)\cap A_{w{p}}=\emptyset$, we have $a\in[a_1,{p}]\cup [a_2, {p}]$, which contradicts~\eqref{eq:lem_bisector}.
\end{proof}

\begin{lemma}\label{lem:bisector}
	Let $B$ be an open disk in $\R^2$ and $w, {p}\in \cl{B}, w\neq {p}$. We write $B_{w{p}}$ for the open disk with diameter $(w,{p})$. Let $a\in \R^2\setminus{\br{\br{B\cap B_{w{p}}}\cup\set{{p}}}}$ such that $\enorm{a-{p}}\leq\enorm{a-w}$ and $\gamma(w, {p}, a)\leq\pi/2$. Then
	\[
	\sin\gamma(w, {p}, a)\geq\frac{\sqrt{2}}{4}\frac{\enorm{w-{p}}}{\diam B}.
	\]
\end{lemma}

\begin{proof}
	Let $A_{w{p}}$ be the set from Lemma~\ref{fact:bisector} applied to $w, {p}$ and $\mc{K}=\cl{B}\cap \cl{B}_{w{p}}$ and let $a_{w{p}}\in A_{w{p}}$ be the minimiser provided there. Note that, by the assumptions, $a\in A_{w{p}}$ as well, and thus, $\gamma(w, {p}, a)\geq\gamma(w, {p}, a_{w{p}})$. Therefore, since $\gamma(w, {p}, a)\leq\pi/2$, to lower bound $\sin\gamma(w, {p}, a)$, we only need to bound $\gamma(w, {p}, a_{w{p}})$ below.
	
	If $a_{w{p}}\notin \partial B$, then by Lemma~\ref{fact:bisector}, $a_{w{p}}\in \partial B_{w{p}}$ and $\gamma(w,{p},a)\geq\pi/4=\gamma(w, {p}, a_{w{p}})$. Thus, 
	\[\sin\gamma(w,{p},a)\geq\frac{\sqrt{2}}{2}>\frac{\sqrt{2}}{4}\geq\frac{\sqrt{2}}{4}\frac{\enorm{w-{p}}}{\diam B}.\]
	
	We can now assume that $a_{w{p}}\in \partial B$. Without loss of generality, we assume that $B$ is centred at $\mb{0}$. Let $b$ stand for the midpoint of $[w,{p}]$, $d:=\enorm{w-{p}}$ and $r:=\diam B /2$. By the reverse triangle inequality, we see that $\enorm{b-a_{w{p}}}\geq r-\enorm{b}\geq r-\sqrt{r^2-(d/2)^2}$.
	
	On the other hand,
	\[
	\sin\gamma(w, {p}, a_{w{p}})=\frac{\enorm{b-a_{w{p}}}}{\enorm{{p}-a_{w{p}}}}=\frac{\enorm{b-a_{w{p}}}}{\sqrt{\enorm{b-a_{w{p}}}^2+\br{\frac{d}{2}}^2}}.
	\]
	Since the function $t\mapsto \frac{t}{\sqrt{t^{2}+\br{\frac{d}{2}}^{2}}}$ is monotone increasing on $(0,\infty)$, we may apply the lower bound $\enorm{b-a_{w{p}}}\geq r-\sqrt{r^2-(d/2)^2}$ to get
	\begin{align*}
		\sin\gamma(w, {p}, a_{w{p}})&\geq\frac{r-\sqrt{r^2-\br{\frac{d}{2}}^2}}{\sqrt{\br{r-\sqrt{r^2-\br{\frac{d}{2}}^2}}^2+\br{\frac{d}{2}}^2}}=\frac{1-\sqrt{1-\br{\frac{d}{2r}}^2}}{\sqrt{\br{1-\sqrt{1-\br{\frac{d}{2r}}^2}}^2+\br{\frac{d}{2r}}^2}}\\
		&\geq\frac{\frac{1}{2}\br{\frac{d}{2r}}^2}{\sqrt{1-\br{1-\br{\frac{d}{2r}}^2}+\br{\frac{d}{2r}}^2}}=\frac{\sqrt{2}}{4}\frac{d}{2r},
	\end{align*}
	where we used that for every $\zeta\in[0, 1]$ it holds that $1-\sqrt{\zeta}\leq\sqrt{1-\zeta}\leq 1-\frac{\zeta}{2}$. 
\end{proof}

Now we are ready for the proof of Theorem~\ref{thm:ext_to_strip}.
\begin{proof}[Proof of Theorem~\ref{thm:ext_to_strip}]
	First we argue that we may assume $h=1$: Suppose the theorem holds if $h=1$ and we are given $L\geq 1$, $h>0$ and an $L$-bilipschitz mapping $f\colon \R\times\set{0,h}\to \R^{2}$. Then let $g\colon \R\times\set{0,1}\to\R^{2}$ be defined by $g(x)=\frac{1}{h}f(hx)$ for all $x\in \R\times\set{0,1}$. The version of the theorem for $h=1$ provides a $10^{19}L^{7}$-bilipschitz extension $G\colon \R\times\set{0,1}\to \R^{2}$ of $g$. Then the mapping $F\colon \set{0,h}\to\R^{2}$ defined by $F(x)=hG\br{\frac{x}{h}}$ is a bilipschitz extension of $f$, verifying Theorem~\ref{thm:ext_to_strip}.

	We may now assume that $h=1$ and fix a polynomial $q(L)$, whose precise value will be specified later. We also write $\mf{L}_{0}:=\R\times\set{0}$ and $\mf{L_{1}}:=\R\times\set{1}$. We start by selecting a sequence $\br{\bar{x}_k}_{k\in\Z}\subset\mf{L}_{0}$ so that the points are increasing in their first coordinate and with distance exactly $q(L)$ between consecutive points. Our aim now is to come up with two sequences $\br{x_k}_{k\in\Z}\subset\mf{L}_{0}$ and $\br{y_k}_{k\in\Z}\subset\mf{L}_{1}$ such that the segments $[x_k, y_k]$ are pairwise non-crossing, and moreover, so that we can extend $f$ to each of these segments in a bilipschitz way simply by interpolating from $f(x_k)$ to $f(y_k)$ linearly, for every $k\in\Z$.
	
	To this end, we  first define $B_k:=B(f(\bar{x}_k), L)$. We want to make sure that these disks are pairwise disjoint, so we require that
	\begin{equation}\label{eq:strip_polyn_req_I}
		q(L)>2L^2.
	\end{equation}
	
	Fix $k\in\Z$ and observe that $\cl{B}_k$ intersects $f(\mf{L}_{1})$. By compactness, there are two points $x_k\in\mf{L}_{0}$ and $y_k\in\mf{L}_{1}$ such that $f(x_k), f(y_k)\in \cl{B}_k$ and
	\begin{equation}\label{eq:dist_fxk_to_fyk}
		\enorm{f(x_k)-f(y_k)}=\dist(f(\mf{L}_{0})\cap \cl{B}_k, f(\mf{L}_{1})\cap \cl{B}_k)\leq L.
	\end{equation}
	As a consequence, $[f(x_k), f(y_k)]$ meets $\img(f)$ only at the endpoints. 
	This allows us to define an injective mapping $F_k\colon \mf{L}_{0}\cup\mf{L}_{1}\cup[x_k, y_k]\to\R^2$ extending $f$ by interpolating linearly between $f(x_k)$ and $f(y_k)$ along $[x_k, y_k]$.
	For a later use, we also record that
	\begin{equation}\label{eq:dist_xk_and_yk_to_barxk}
		\enorm{x_k-\bar{x}_k}, \enorm{y_k-\bar{x}_k}, \enorm{x_k-y_k}\leq L^2.
	\end{equation}

	\begin{claim}\label{claim:F_k_bilip}
		The mapping $F_k$ satisfies $\lip(F_k)\leq 2L^3$ and $\lip(F_k^{-1})\leq 12L^3$. 
	\end{claim}
	\begin{proof}
		The points $x_k, y_k$ lie on two parallel lines $\mf{L}_{0}, \mf{L}_{1}$ with distance $1$. Therefore, letting $\alpha_{k}$ denote the smallest of the angles between $[x_{k},y_{k}]$ and $\mf{L}_{0}$ (or $\mf{L}_{1}$), by~\eqref{eq:dist_xk_and_yk_to_barxk} we have
		\begin{equation}\label{eq:sin_alpha_k}
			\sin\alpha_k=\frac{1}{\enorm{x_k-y_k}}\geq\frac{1}{L^2}.
		\end{equation}
		
		The mapping $\rest{F_k}{[x_k, y_k]}$ is affine and thus its bilipschitz constants are realised on the pair $x_k, y_k$, where $F_k$ is equal to $f$. Because $f$ is $L$-bilipschitz, we have that $\rest{F_k}{[x_k, y_k]}$ is $L$-bilipschitz as well.
		
		Since we know that $\rest{F_k}{\mf{L}_{0}\cup\mf{L}_{1}}$ is $L$-bilipschitz, it is enough to bound the bilipschitz constants of $\rest{F_k}{\mf{L}_i\cup[x_k, y_k]}$, $i\in\set{0,1}$, separately. We discuss only the case $i=0$, as the other one is the same.
		
		We again use Lemma~\ref{lem:gluing_lip}. To this end, given $x\in\mf{L}_{0}$ and $y\in[x_k, y_k]$, we set $z(x,y):=x_k$. By the sine theorem and \eqref{eq:sin_alpha_k}, we get that
		\begin{equation}
			\enorm{x-z(x,y)}+\enorm{y-z(x,y)}\leq \frac{2\enorm{x-y}}{\sin\alpha_k}\leq 2L^2\enorm{x-y},
		\end{equation}
		which proves the first part of the claim.
		
		Now we are given $a\in f(\mf{L}_{0})$ and $b\in[f(x_k), f(y_k)]$. We can assume that $b\neq f(x_k), f(y_k)$. We set $z(a,b):=f(x_k)$ if $\enorm{a-f(x_k)}\leq\enorm{a-f(y_k)}$; in the opposite case, we set $z(a,b):= f(y_k)$.

		We denote the angle in the triangle $z(a, b) a b$ at the vertex $z(a, b)$ by $\gamma_z$. If $\gamma_z \geq \frac{\pi}{2}$, then $\enorm{a-b}$ is the longest side of the triangle, and hence, $\enorm{a-z(a,b)}+\enorm{b-z(a,b)}\leq 2\enorm{a-b}$.

		When $\gamma_z < \frac{\pi}{2}$, we set $z:=z(a,b)=f(x_{k})$ and $w:=f(y_{k})$ (or the other way around) and note that $\gamma_z$ is also the angle at $z(a,b)$ in the triangle $af(x_k)f(y_k)$. Moreover, we observe that $a\notin B_k\cap B_{wz}$, as otherwise $a$ would be a point of $f(\mf{L}_{0})\cap B_k$ closer to $f(y_k)\in f(\mf{L}_{1})\cap \cl{B}_{k}$ than $f(x_{k})$, contradicting the definition of $x_k$ and $y_k$; see \eqref{eq:dist_fxk_to_fyk}. Thus, we can apply Lemma~\ref{lem:bisector} to $a, w, p=z$ and $B=B_{k}$ and obtain that
		\begin{equation}\label{eq:sin_gamma_lower_bound}
			\sin\gamma_z=\sin\gamma(w, z, a)\geq\frac{\sqrt{2}}{4}\frac{\enorm{f(x_k)-f(y_k)}}{\diam B_k}.
		\end{equation}

		Now we use again the sine theorem and deduce $\frac{2\enorm{a-b}}{\sin\gamma_z}\geq\enorm{a-z(a,b)}+\enorm{b-z(a,b)}$. By the bilipschitz property of $f$, $\frac{\enorm{f(x_k)-f(y_k)}}{\diam B_k}\geq\frac{1/L}{2L}=1/2L^2$. Combining it with \eqref{eq:sin_gamma_lower_bound}, we have shown that given any $a\in f(\mf{L}_{0})$ and $b\in[f(x_k), f(y_k)]$,
		\begin{equation*}
			\enorm{a-z(a,b)}+\enorm{b-z(a,b)}\leq 8\sqrt{2}L^2\enorm{a-b}\leq 12L^2\enorm{a-b}.
		\end{equation*}
		By Lemma~\ref{lem:gluing_lip}, we conclude that $\lip(F_k^{-1})\leq 12L^3$.
	\end{proof}

	By definition, $\enorm{\cl{x}_k-\cl{x}_{k+1}}=q(L)$ for $k\in\Z$, and thus, by \eqref{eq:strip_polyn_req_I} and \eqref{eq:dist_xk_and_yk_to_barxk}, the relative order of both $\br{x_k}_{k\in\Z}$ and $\br{y_k}_{k\in\Z}$ along $\mf{L}_{0}$ and $\mf{L}_{1}$, respectively, is the same as that of $\br{\bar{x}_k}_{k\in\Z}$. Therefore, the segments $[x_k, y_k]$ are pairwise non-crossing.
	
	Next, we would like to use Lemma~\ref{lem:gluing_lip} to show that the mappings $\br{F_k}_{k\in\Z}$ combine to give a bilipschitz mapping.
	
	\begin{claim}\label{claim:F_ks_together_bilip}
		Let $G\colon\mf{L}_{0}\cup\mf{L}_{1}\cup\bigcup_{k\in\Z}[x_k, y_k]\to\R^2$ be the mapping extending $F_k$ for every $k\in\Z$.
		Then $\lip(G)\leq 4L^3$ and $\lip(G^{-1})\leq 24L^3$.
	\end{claim}
	\begin{proof}
		We note that $F_k\cup F_l$ and $F_k^{-1}\cup F_l^{-1}$ are well-defined according to Definition~\ref{def:gluing} for every $k, l\in\Z$. As $G$ is $L'$-bilipschitz if and only if each $F_k\cup F_l$ is $L'$-bilipschitz for every $k, l\in\Z$, it is enough to argue that there is $C>0$ such that for every $k, l\in\Z$ the mappings $F_k, F_l$ as well as the mappings $F^{-1}_k, F^{-1}_l$ satisfy the assumptions of Lemma~\ref{lem:gluing_lip} with the constant $C$.
		In fact, to obtain the claimed Lipschitz constants, in light of Claim~\ref{claim:F_k_bilip} it suffices to show that we can choose $C=2$. For this, we additionally assume
		\begin{equation}\label{eq:strip_polyn_req_II}
			q(L)\geq 4L^2.
		\end{equation}
		
		We fix $k<l, k,l\in\Z$, $p_k\in[x_k, y_k]$ and $p_l\in[x_l, y_l]$. We set $z(p_k, p_l):=\bar{x}_k$ and argue that this choice satisfies the assumptions of Lemma~\ref{lem:gluing_lip} for $F_k$ and $F_l$.
		
		By the triangle inequality, we can write
		\begin{align}\label{eq:p_k_I}
			\begin{split}
				\enorm{p_k-z(p_k, p_l)}+\enorm{p_l-z(p_k, p_l)}&\leq\enorm{p_k-\bar{x}_k}+\enorm{p_k-\bar{x}_k}+\enorm{p_k-p_l}.
			\end{split}
		\end{align}
		On the other hand, by \eqref{eq:dist_xk_and_yk_to_barxk}, \eqref{eq:strip_polyn_req_II} and the definition of $\br{\bar{x}_i}_{i\in\Z}$ (applied in this order), we have 
		\begin{align}\label{eq:p_k_II}
			\begin{split}
				2\enorm{p_k-\bar{x}_k}\leq 2L^2\leq q(L)-2L^2\leq\enorm{p_k-p_l}.
			\end{split}
		\end{align}
		Combining \eqref{eq:p_k_I} and \eqref{eq:p_k_II}, we directly obtain that
		\begin{equation*}
			\enorm{p_k-z(p_k, p_l)}+\enorm{p_l-z(p_k, p_l)}\leq 2\enorm{p_k-p_l}.
		\end{equation*}

		For $F_k^{-1}$ and $F_l^{-1}$ we choose $z:=z(f(p_k), f(p_l)):=f(\bar{x}_k)$ and use exactly the same method as for $F_k, F_l$ to show that the assumptions of Lemma~\ref{lem:gluing_lip} are satisfied with $C=2$. Thus, we write only the relevant inequalities without further justification:
		\begin{equation*}
			\enorm{f(p_k)-z}+\enorm{f(p_l)-z}\leq\enorm{f(p_k)-f(\bar{x}_k)}+\enorm{f(p_k)-f(\bar{x}_k)}+\enorm{f(p_k)-f(p_l)}.
		\end{equation*}
		Instead of \eqref{eq:dist_xk_and_yk_to_barxk} we use \eqref{eq:dist_fxk_to_fyk} and the bilipschitz property of $f$:
		\begin{equation*}
			2\enorm{f(p_k)-f(\bar{x}_k)}\leq 2L\leq \frac{q(L)}{L}-2L\leq \enorm{f(p_k)-f(p_l)}.
		\end{equation*}
		Altogether, we get that 
		\begin{align*}
			\enorm{f(p_k)-z}+\enorm{f(p_l)-z}&\leq2\enorm{f(p_k)-f(p_l)}.\qedhere
		\end{align*}
	\end{proof}
	
	We summarise where we currently stand. For every $k\in\Z$, we have a trapezium $P_k$ with vertices $x_k, x_{k+1}, y_k$ and $y_{k+1}$. Two consecutive trapezia $P_k$ and $P_{k+1}$ share an edge. Additionally, we have a bilipschitz extension $G$ of $f$ to $\bigcup_{k\in\Z} \partial P_k$.

	\begin{claim}\label{claim:strip_trapezia_bilip}
		For every $k\in\Z$, 
		there is a bilipschitz bijection $T_k\colon [0, 1]^2\to P_k$ with $\bilip(T_k)\leq\frac{7q(L)}{4}$.
	\end{claim}
	\begin{proof}
		We fix $k\in\Z$. In line with the notation of Lemma~\ref{lem:trapezium_bilip}, we call by $b_1\leq b_2$ the lengths of the bases of $P_k$ and by $l_1\leq l_2$ the lengths of the legs of $P_k$. The height of $P_k$ is $1$. By the definition of $\br{\bar{x}_i}_{i\in\Z}$ and \eqref{eq:dist_xk_and_yk_to_barxk}, the bases of $P_k$ satisfy
		\begin{equation*}
			q(L)-2L^2\leq b_1\leq b_2\leq q(L)+2L^2
		\end{equation*}
		and the legs of $P_k$ satisfy 
		\begin{equation*}
			1\leq l_1\leq l_2\leq L^2.
		\end{equation*}
		By Lemma~\ref{lem:trapezium_bilip} there is a bilipschitz bijection $T_k\colon [0, 1]^2\to P_k$ satisfying some explicitly stated bounds on $\lip(T_{k})$ and $\lip(T_{k}^{-1})$ in terms of the quantities $b_{1},b_{2},l_{2}$. From the bound on $\lip(T_{k})$ and \eqref{eq:strip_polyn_req_II}, we obtain
		\[\lip(T_k)\leq q(L)+2L^2+L^2\leq\frac{7q(L)}{4}.\]
		
		Next we show that $\lip(T_k^{-1})\leq 4$. This will complete the proof, since $\frac{7q(L)}{4}\geq 7L^2\geq 7$ by \eqref{eq:strip_polyn_req_II}. In view of the bound on $\lip(T_{k}^{-1})$ from Lemma~\ref{lem:trapezium_bilip}, it suffices to argue that 
		\[4b_1\geq \sqrt{b_2^2+\frac{3l_2^2}{2}}.\]
		By the upper bounds on $b_2$ and $l_2$ stated above and \eqref{eq:strip_polyn_req_II}, we get that 
		\[b_2^2+\frac{3l_2^2}{2}\leq (q(L)+2L^2)^2+\frac{3(L^2)^2}{2}\leq\br{\frac{3q(L)}{2}}^2+\frac{3}{2}\br{\frac{q(L)}{4}}^2<3q^2(L).\]
		On the other hand, the lower bound on $b_1$ and \eqref{eq:strip_polyn_req_II} imply 
		\[16b_1^2\geq 16\br{q(L)-2L^2}^2\geq 16\br{\frac{q(L)}{2}}^2=4q^2(L).\]
		A comparison of the two inequalities establishes the claimed upper bound on $\lip(T_k^{-1})$. 
	\end{proof}
	
	As a direct consequence of Claims~\ref{claim:F_ks_together_bilip} and \ref{claim:strip_trapezia_bilip} combined with Proposition~\ref{p:generalised_DP}, we can extend each $\rest{G}{\partial P_k}$ defined in Claim~\ref{claim:F_ks_together_bilip} to a bilipschitz mapping $G_{k}\colon P_k\to\R^2$ and directly compute the following upper bound for its bilipschitz constant:
	\begin{equation}\label{eq:bilipGk}
		\bilip(G_k)\leq 2\cdot 10^{16}\cdot\br{\frac{7q(L)}{4}}^{2}\cdot 24L^{3}\leq 2\cdot 10^{18}L^{3}q(L)^{2}.
	\end{equation}
	
	It remains to argue that the mapping $F\colon \R\times[0, 1]\to\R^2$ defined as $\rest{F}{P_k}\equiv G_k$ for every $k\in\Z$ is bilipschitz as well. First, consider an arbitrary $k\in\Z$ and recall the mapping $G$ from Claim~\ref{claim:F_ks_together_bilip}. Since $\partial\dom(G_k)=\partial P_k\subset\dom(G)$, Corollary~\ref{cor:gluing_bilip} applies to $G_k$ and $G$, so we get that $G_k\cup G$ is bilipschitz with $\bilip(G_k\cup G)=\max\set{\bilip(G), \bilip(G_k)}$. Second, for every $k, l\in\Z$, Corollary~\ref{cor:gluing_bilip} applies to $G_k\cup G, G_l\cup G$ verifying that $\bilip(\rest{F}{P_k\cup P_l})\leq\max\set{\bilip(G_k\cup G), \bilip(G_l\cup G)}$. Therefore, $F$ is bilipschitz and $\bilip(F)$ obeys the bound of \eqref{eq:bilipGk} as well.
	
	Finally, noting that the strictest requirement on $q(L)$ has been \eqref{eq:strip_polyn_req_II}, setting $q(L):=4L^2$ and plugging this value into the bound of \eqref{eq:bilipGk} finishes the proof.
\end{proof}
\bibliographystyle{plain}
\bibliography{citations}
\end{document}